\documentclass[twoside]{amsart}
\usepackage{ucs}
\usepackage[utf8x]{inputenc}
\usepackage[english]{babel}
\usepackage{amsmath,amsthm}
\usepackage{amssymb}
\usepackage[all]{xy}
\usepackage{geometry}

\usepackage{amsfonts,mathrsfs}
\usepackage{enumitem}
\usepackage[dpi=1270,PostScript=dvips,heads=LaTeX]{diagrams}
\diagramstyle{small,labelstyle=\scriptstyle}
\usepackage{hyperref}

\newcommand{\R}{\mathbb{R}}
\newcommand{\C}{\mathbb{C}}
\newcommand{\Q}{\mathbb{Q}}
\newcommand{\Z}{\mathbb{Z}}
\newcommand{\N}{\mathbb{N}}

\newcommand{\ra}{\rightarrow}

\newtheorem{thm}{Theorem}
\newtheorem{df}[thm]{Definition}
\newtheorem{prop}[thm]{Proposition}
\newtheorem{lem}[thm]{Lemma}
\newtheorem{cor}[thm]{Corollary}
\newtheorem{conj}[thm]{Conjecture}
\newtheorem{ex}[thm]{Example}

\theoremstyle{remark}
\newtheorem{rk}[thm]{Remark}
\newtheorem{exc}[thm]{Exercise}

\newcommand{\fl}{f_{\ast}}
\newcommand{\fu}{f^{\ast}}
\newcommand{\decal}[1]{\lbrack #1 \rbrack}
\newcommand{\PP}{\mathbb{P}}

\DeclareMathOperator{\im}{Im}
\DeclareMathOperator{\codim}{codim}

\newcommand{\QHS}{\text{$\QQ$-HS}}

\newcommand{\shH}{\mathcal{H}}

\DeclareMathOperator{\MT}{MT}

\DeclareMathOperator{\Ad}{Ad}

\newcommand{\End}{\operatorname{End}}
\newcommand{\HdR}[1]{H^{#1}}
\newcommand{\Xan}{X^{\mathrm{an}}}
\newcommand{\Dan}{D^{\mathrm{an}}}
\newcommand{\Zan}{Z^{\mathrm{an}}}
\newcommand{\Ban}{B^{\mathrm{an}}}
\newcommand{\Fan}{\shF^{\mathrm{an}}}
\newcommand{\Ean}{\shE^{\mathrm{an}}}
\newcommand{\Lan}{\shL^{\mathrm{an}}}
\newcommand{\fan}{f^{\mathrm{an}}}
\newcommand{\piu}{\pi^{\ast}}

\newcommand{\classdR}[1]{\lbrack #1 \rbrack}
\newcommand{\cdR}{c}
\newcommand{\shL}{\mathscr{L}}
\newcommand{\famA}{\mathcal{A}}
\DeclareMathOperator{\Gal}{Gal}
\DeclareMathOperator{\Aut}{Aut}
\newcommand{\Hom}{\operatorname{Hom}}

\newcommand{\mlie}{\mathfrak{m}}
\newcommand{\Eun}{E^{\times}}
\newcommand{\Fun}{F^{\times}}

\newcommand{\Qdeg}[1]{\lbrack #1 : \QQ \rbrack}
\renewcommand{\sb}{\bar{s}}
\newcommand{\TrEQ}{\Tr_{E/\QQ}}
\newcommand{\NmEF}{\Nm_{E/F}}
\DeclareMathOperator{\Tr}{Tr}
\DeclareMathOperator{\Nm}{Nm}
\newcommand{\disc}{\operatorname{disc}}
\newcommand{\phizeta}{\varphi_{\zeta}}
\newcommand{\wed}[2]{\bigwedge\nolimits_{#1}^{#2}}
\newcommand{\Grass}{\operatorname{Grass}}
\newcommand{\MgdN}{\mathcal{M}_{g,d,N}}
\newcommand{\Dsp}{D^{\mathrm{sp}}}
\newcommand{\abs}[1]{\lvert #1 \rvert}
\newcommand{\tensor}{\otimes}

\newcommand{\ZZ}{\mathbb{Z}}
\newcommand{\QQ}{\mathbb{Q}}
\newcommand{\RR}{\mathbb{R}}
\newcommand{\CC}{\mathbb{C}}
\newcommand{\HH}{\mathbb{H}}
\newcommand{\menge}[2]{\bigl\{ \thinspace #1 \thinspace\thinspace \big\vert%
\thinspace\thinspace #2 \thinspace \bigr\}}
\newcommand{\define}[1]{\emph{#1}}
\newcommand{\shf}[1]{\mathscr{#1}}
\newcommand{\OX}{\shO_X}
\newcommand{\shO}{\shf{O}}
\newcommand{\shF}{\shf{F}}
\newcommand{\shE}{\shf{E}}
\newcommand{\into}{\hookrightarrow}

\DeclareMathOperator{\GL}{GL}

\DeclareMathOperator{\sgn}{sgn}

\newcommand{\class}[1]{\lbrack #1 \rbrack}
\newcommand{\Spec}{\operatorname{Spec}}
\newcommand{\argbl}{-}
\newcommand{\Mmod}{\mathcal{M}}

\title{Notes on absolute Hodge classes}

\author{Fran\c{c}ois Charles}
\address{D\'epartement de Math\'ematiques et Applications, \'Ecole Normale Sup\'erieure, 45, rue d'Ulm, 75005 Paris,
France}
\email{francois.charles@ens.fr}
\author{Christian Schnell}
\address{Department of Mathematics, Statistics and Computer Science, University of Illinois at Chicago, 851 S. Morgan
Street, Chicago, IL 60607, USA}
\email{cschnell@math.uic.edu}
\date{}
\begin{document}

\maketitle

\section*{Introduction}

Absolute Hodge classes first appear in Deligne's proof of the Weil conjectures for K3 surfaces in \cite{DelK3} and are
explicitly introduced in \cite{Del82}. The notion of absolute Hodge classes in the singular cohomology of a smooth
projective variety stands between that of Hodge classes and classes of algebraic cycles. While it is not known whether
absolute Hodge classes are algebraic, their definition is both of an analytic and arithmetic nature.

The paper \cite{DelK3} contains one of the first appearances of the notion of motives, and is among the first
unconditional applications of motivic ideas. Part of the importance of the notion of absolute Hodge classes is indeed to
provide an unconditional setting for the application of motivic ideas. The papers \cite{DelK3}, \cite{DMOS} 
and \cite{An}, among others, give examples of this train of thoughts. The book \cite{Jan} develops a theory of mixed
motives based on absolute Hodge classes.

\bigskip

In these notes, we survey the theory of Hodge classes. The first section of these notes deals with algebraic de Rham
cohomology and cycle classes. As proved by Grothendieck, the singular cohomology groups of a complex algebraic variety
can be computed using suitable algebraic de Rham complexes. This provides an algebraic device for computing topological
invariants of complex algebraic varieties. 

The preceding construction is the main tool behind the definition of Hodge classes, which is the
object of section 2. Indeed, this allows to consider Galois actions on the singular cohomology groups with complex
coefficients of a complex algebraic variety. In section 2, we discuss the definition of absolute Hodge classes. We
try to investigate two aspects of this subject. The first one pertains to the Hodge conjecture. Absolute Hodge classes
make it possible to shed some light on the problem of the algebraicity of Hodge classes, and allow us to isolate the
number-theoretic content of the Hodge conjecture. The second aspect we hint at is the motivic meaning of absolute Hodge
classes. While we do not discuss the construction of motives for absolute Hodge classes as in \cite{DMOS}, we show
various functoriality and semi-simplicity properties of absolute Hodge classes which lie behind the more general
motivic constructions cited above. We try to phrase our results so as to get results and proofs which are valid for
Andr\'e's theory of motivated cycles as in \cite{An}. While we do not define motivated cycles, some of our proofs are
very much inspired by this paper.

The third section deals with variational properties of absolute Hodge classes. After stating the variational Hodge
conjecture, we prove Deligne's principle B as in \cite{Del82} which one of the main technical tools of the paper. In
the remainder of the section, we discuss consequences of the algebraicity of Hodge bundles and of the Galois action on
relative de Rham cohomology. Following \cite{Voabs}, we investigate the meaning of the theorem of
Deligne-Cattani-Kaplan on the algebraicity of Hodge loci, see \cite{CDK}, and discuss the link between Hodge classes
being absolute and the field of definition of Hodge loci.

The last two sections are devoted to important examples of absolute Hodge classes. Section 4 discusses the Kuga-Satake
correspondence following Deligne in \cite{DelK3}. In section 5, we give a full proof of Deligne's theorem which states
that Hodge classes on abelian varieties are absolute, see \cite{Del82}.

In writing these notes, we did not seek efficiency at all costs. Indeed, we did not necessarily prove properties of
absolute Hodge cycles in the shortest way possible, but we rather chose to emphasize a variety of techniques and ideas.

\subsection*{Acknowledgements} This text is an expanded version of five lectures given at the ICTP summer school
on Hodge theory in Trieste in June 2010. The first two lectures were devoted to absolute Hodge cycles and arithmetic
aspects of the Hodge conjecture. The remaining three lectures outlined Deligne's proof that every Hodge class on an
abelian variety is an absolute Hodge class. We would like to thank the organizers for this very nice and fruitful summer
school.

Claire Voisin was supposed to give these lectures in Trieste, but she could not attend. It would be hard to acknowledge
enough the influence of her work on these notes and the lectures we gave. We are grateful for her allowing us to use
the beautiful survey \cite{VoHodge} for our lectures and thank her sincerely. We also want to thank Matt Kerr for
giving one of the lectures. 

\bigskip

\tableofcontents


\section{Algebraic de Rham cohomology}

Shortly after Hironaka's paper on resolutions of singularities had appeared,
Grothendieck observed that the cohomology groups of a complex algebraic variety can
be computed algebraically. More precisely, he showed that on a nonsingular $n$-dimensional algebraic
variety $X$ (of finite type over the field of complex numbers $\CC$), the
hypercohomology of the algebraic de Rham complex
\[
	\OX \to \Omega_{X/\CC}^1 \to \dotsb \to \Omega_{X/\CC}^n 
\]
is isomorphic to the singular cohomology $H^{\ast}(\Xan, \CC)$ of the complex manifold
corresponding to $X$. Grothendieck's theorem makes it possible to ask arithmetic
questions in Hodge theory, and is the founding stone for the theory of
absolute Hodge classes. In this lecture, we briefly review Grothendieck's theorem, as
well as the construction of cycle classes in algebraic de Rham cohomology.

\subsection{Algebraic de Rham cohomology}

We begin by describing algebraic de Rham cohomology in a more general setting. Let
$X$ be a nonsingular quasi-projective variety, defined over a field $K$ of
characteristic zero. This means that we have a morphism $X \to \Spec K$, and we let
$\Omega_{X/K}^1$ denote the sheaf of K\"ahler differentials on $X$. We also define $\Omega_{X/K}^i =
\bigwedge^i \Omega_{X/K}^1$.

\begin{df}
The \define{algebraic de Rham cohomology} of $X \to K$ consists of the $K$-vector
spaces
\[
	\HdR{i}(X/K) = \HH^i \bigl( \OX \to \Omega_{X/K}^1 \to \dotsb \to \Omega_{X/K}^n \bigr),
\]
where $n = \dim X$.
\end{df}

This definition is compatible with field extensions, for the following reason. Given
a field extension $K \subseteq L$, we let $X_L = X \times_{\Spec K} \Spec L$ denote
the variety obtained from $X$ by extension of scalars. Since $\Omega_{X_L/L}^1 \simeq
\Omega_{X/K}^1 \tensor_K L$, we obtain $\HdR{i}(X_L/L) \simeq \HdR{i}(X/K) \tensor_K L$. 

The algebraic de Rham complex $\Omega_{X/K}^{\bullet}$ is naturally filtered by the
subcomplexes $\Omega_{X/K}^{\bullet \geq p}$, and we shall denote the
induced filtration on its hypercohomology by
\[
	F^p \HdR{i}(X/K) = \im \bigl(
		\HH^i(\Omega_{X/K}^{\bullet \geq p} \to \Omega_{X/K}^{\bullet} \bigr)
\]
and refer to it as the Hodge filtration. We can now state Grothendieck's comparison theorem.

\begin{thm}[Grothendieck]
Let $X$ be a nonsingular projective variety over $\CC$, and let $\Xan$ denote the
associated complex manifold. Then there is a canonical isomorphism
\[
	\HdR{i}(X/\CC) \simeq H^i(\Xan, \CC),
\]
and under this isomorphism, $F^p \HdR{i}(X/\CC) \simeq F^p H^i(\Xan, \CC)$ gives the
Hodge filtration on singular cohomology.
\end{thm}

\begin{proof}
The theorem is a consequence of the GAGA theorem of Serre. Let $\shO_{\Xan}$ denote
the sheaf of holomorphic functions on the complex manifold $\Xan$. We then have a
morphism $\pi \colon (\Xan, \shO_{\Xan}) \to (X, \OX)$ of locally ringed spaces. 
For any coherent sheaf $\shF$ on $X$, the associated coherent analytic sheaf on
$\Xan$ is given by $\Fan = \piu \shF$, and according to Serre's theorem, $H^i(X,
\shF) \simeq H^i(\Xan, \Fan)$. 

It is easy to see from the local description of the sheaf of K\"ahler
differentials that $(\Omega_{X/\CC}^1)^{\mathrm{an}} = \Omega_{\Xan}^1$. This implies
that $H^q(X, \Omega_{X/\CC}^p) \simeq H^q(\Xan, \Omega_{\Xan}^p)$ for all $p,q \geq
0$. Now pullback via $\pi$ induces homomorphisms
$\HH^i(\Omega_{X/\CC}^{\bullet}) \to \HH^i(\Omega_{\Xan}^{\bullet})$, which are
isomorphism by Serre's theorem. Indeed, the groups on the left are computed by a
spectral sequence with $E_2^{p,q}(X) = H^q(X, \Omega_{X/\CC}^p)$, and the groups on
the right by a spectral sequence with terms $E_2^{p,q}(\Xan) = H^q(\Xan,
\Omega_{\Xan}^p)$, and the two spectral sequences are isomorphic starting from the 
$E_2$-page. By the Poincar\'e lemma, the holomorphic de Rham complex $\Omega_{\Xan}^{\bullet}$ is
a resolution of the constant sheaf $\CC$, and therefore $H^i(\Xan, \CC) \simeq
\HH^i(\Omega_{\Xan}^{\bullet})$. Putting everything together, we obtain a canonical
isomorphism
\[
	\HdR{i}(X/\CC) \simeq H^i(\Xan, \CC).
\]

Since the Hodge filtration on $H^i(\Xan, \CC)$ is induced by the naive filtration on the complex
$\Omega_{\Xan}^{\bullet}$, the second assertion follows by the same argument.
\end{proof}

\begin{rk}
A similar result holds when $X$ is nonsingular and quasi-projective. Using resolution
of singularities, one can find a nonsingular variety $\overline{X}$ and a divisor with normal
crossing singularities, such that $X = \overline{X} - D$. Using differential forms with at
worst logarithmic poles along $D$, one still has
\[
	H^i(\Xan, \CC) \simeq \HH^i(\Omega_{\overline{X}^{an}}^{\bullet}(\log \Dan))
		\simeq \HH^i(\Omega_{\overline{X}/\CC}^{\bullet}(\log D));
\]
under this isomorphism, the Hodge filtration is again induced by the naive filtration
on the logarithmic de Rham complex $\Omega_{\overline{X}^{an}}^{\bullet}(\log \Dan)$. Since algebraic differential forms
on $X$
have at worst poles along $D$, it can further be shown that those groups are
still isomorphic to $\HdR{i}(X/\CC)$.

The general case of a possibly singular quasi-projective variety is dealt with in \cite{DelH3}. It
involves the previous construction together with simplicial techniques.
\end{rk}

Now suppose that $X$ is defined over a subfield $K \subseteq \CC$. Then the complex
vector space $H^i(\Xan, \CC)$ has two additional structures: a $\QQ$-structure,
coming from the universal coefficients theorem
\[
	H^i(\Xan, \CC) \simeq H^i(\Xan, \QQ) \tensor_{\QQ} \CC,
\]
and a $K$-structure, coming from Grothendieck's theorem
\[
	H^i(\Xan, \CC) \simeq \HdR{i}(X/K) \tensor_K \CC.
\]
In general, these two structures are not compatible with each other. It should be
noted that the Hodge filtration is defined over $K$.

The same construction works in families to show that Hodge bundles and the
Gauss-Manin connection are algebraic. Let $f \colon X \to B$ be a smooth projective
morphism of varieties over $\CC$. For each $i$, it determines a variation of Hodge
structure on $B$ whose underlying vector bundle is
\[
	\shH^i = R^i \fl \QQ \tensor_{\QQ} \shO_{\Ban}
		\simeq \RR^i \fan_{\ast} \Omega_{\Xan/\Ban}^{\bullet}
		\simeq \bigl( \RR^i \fl \Omega_{X/B}^{\bullet} \bigr)^{\mathrm{an}}.
\]
By the relative version of Grothendieck's theorem, the Hodge bundles are given by
\[
	F^p \shH^i \simeq \bigl( \RR^i \fl \Omega_{X/B}^{\bullet \geq p} \bigr)^{\mathrm{an}}.
\]
Katz and Oda have shown  in \cite{KO} that the Gauss-Manin connection $\nabla \colon \shH^i \to
\Omega_{\Ban}^1 \tensor \shH^i$ can also be constructed algebraically. Starting from the
exact sequence
\[
	0 \to \fu \Omega_{B/\CC}^1 \to \Omega_{X/\CC}^1 \to \Omega_{B/\CC}^1 \to 0,
\]
let $L^r \Omega_{X/\CC}^i = \fu \Omega_{B/\CC}^r \wedge \Omega_{X/\CC}^{i-r}$. We get
a short exact sequence of complexes
\[
	0 \to \fu \Omega_{B/\CC}^1 \tensor \Omega_{X/B}^{\bullet-1} \to
		\Omega_{X/\CC}^{\bullet} / L^2 \Omega_{X/\CC}^{\bullet} \to
		\Omega_{X/B}^{\bullet} \to 0,
\]
and hence a connecting morphism
\[
	\RR^i \fl \Omega_{X/B}^{\bullet} \to \RR^{i+1} \fl \bigl( \fu \Omega_{B/\CC}^1
		\tensor \Omega_{X/B}^{\bullet-1} \bigr)
	\simeq \Omega_{B/\CC}^1 \tensor \RR^i \fl \Omega_{X/B}^{\bullet}.
\]
The theorem of Katz-Oda is that the associated morphism between analytic vector bundles is
precisely the Gauss-Manin connection $\nabla$.

For our purposes, the most interesting conclusion is the following: if $f$, $X$, and
$B$ are all defined over a subfield $K \subseteq \CC$, then the same is true for the
Hodge bundles $F^p \shH^i$ and the Gauss-Manin connection $\nabla$. We shall make use
of this fact later when discussing absolute Hodge classes and Deligne's Principle~B.

\subsection{Cycle classes}

Let $X$ be a nonsingular projective variety over $\CC$, of dimension $n = \dim X$. 
Let $Z \subseteq X$ be an algebraic subvariety of codimension $p$. It determines a
cycle class
\[
	\class{\Zan} \in H^{2p} \bigl( \Xan, \QQ(p) \bigr)
\]
in Betti cohomology, as follows: Let $\widetilde{Z}$ be a resolution of singularities of
$Z$, and let $\mu \colon \widetilde{Z} \to X$ denote the induced morphism. By Poincar\'e
duality, the linear functional
\[
	H^{2n-2p}(\Xan, \QQ) \to \QQ, \qquad
		\alpha \mapsto \frac{1}{(2 \pi i)^{n-p}} \int_{\widetilde{Z}^{\mathrm{an}}} \mu^{\ast}(\alpha)
\]
is represented by a unique class $\zeta \in H^{2p}(\Xan, \QQ)$, with the
property that
\[
	 \frac{1}{(2 \pi i)^{n-p}} \int_{\widetilde{Z}^{\mathrm{an}}} \mu^{\ast}(\alpha)
		= \frac{1}{(2 \pi i)^n} \int_{\Xan} \zeta \cup \alpha.
\]
This class is clearly of type $(p,p)$, and therefore a Hodge class. To eliminate the
annoying factor of $(2 \pi i)$, we define the cycle class of $Z$ to be $\class{\Zan} = (2 \pi
i)^p \zeta$, which is now a rational Hodge class in the weight zero Hodge structure $H^{2p}
\bigl( \Xan, \QQ(p) \bigr)$.

An important fact is that one can also define a cycle class
\[
	\classdR{Z} \in F^p \HdR{2p}(X/\CC)
\]
in algebraic de Rham cohomology, and that the following comparison theorem holds.

\begin{thm}
Under the isomorphism $\HdR{2p}(X/\CC) \simeq H^{2p}(\Xan, \CC)$, we have
\[
	\classdR{Z} = \class{\Zan}.
\]
\end{thm}

Consequently, if $Z$ and $X$ are both defined over a subfield $K \subseteq \CC$, then
the cycle class $\class{\Zan}$ is actually defined over the algebraic closure $\bar{K}$.

In the remainder of this lecture, our goal is to understand the construction of the
algebraic cycle class, and where the factor $(2 \pi i)^p$ in the definition of the
cycle class comes from. We shall first look at a nice
special case, due to Bloch. Assume for now that $Z$ is a local complete intersection
of codimension $p$. This means that $X$ can be covered by open sets $U$, with the
property that $Z \cap U = V(f_1, \dotsc, f_p)$ is the zero scheme of $p$ regular
functions $f_1, \dotsc, f_p$. Then $U - Z \cap U$ is covered by the open sets
$D(f_1), \dotsc, D(f_p)$, and
\begin{equation} \label{eq:Bloch}
	\frac{df_1}{f_1} \wedge \dotsb \wedge \frac{df_p}{f_p}
\end{equation}
is a closed $p$-form on $D(f_1) \cap \dotsb \cap D(f_p)$. Using \v{C}ech cohomology,
it determines a class in
\[
	H^{p-1} \bigl( U - Z \cap U, \Omega_{X/\CC}^{p, \mathrm{cl}} \bigr),
\]
where $\Omega_{X/\CC}^{p, \mathrm{cl}}$ is the subsheaf of $\Omega_{X/\CC}^p$
consisting of closed $p$-forms. Since we have a map of complexes $\Omega_{X/\CC}^{p,
\mathrm{cl}} \decal{-p} \to \Omega_{X/\CC}^{\bullet \geq p}$, we get
\[
	H^{p-1} \bigl( U - Z \cap U, \Omega_{X/\CC}^{p, \mathrm{cl}} \bigr)
	\to \HH^{2p-1} \bigl( U - Z \cap U, \Omega_{X/\CC}^{\bullet \geq p} \bigr)
	\to \HH_{Z \cap U}^{2p} \bigl( \Omega_{X/\CC}^{\bullet \geq p} \bigr).
\]
One can show that the image of \eqref{eq:Bloch} in the cohomology group with supports
on the right does not depend on the choice of local equations $f_1, \dotsc, f_p$. (A
good exercise is to prove this for $p=1$ and $p=2$.) It therefore defines a global section of
the sheaf $\shH_Z^{2p}(\Omega_{X/\CC}^{\bullet \geq p})$.
Using that $\shH_Z^i(\Omega_{X/\CC}^{\bullet \geq p}) = 0$ for $i \leq 2p-1$, we get
from the local-to-global spectral sequence that
\[
	\HH_Z^{2p} \bigl( \Omega_{X/\CC}^{\bullet \geq p} \bigr)
	\simeq H^0 \bigl( X, \shH_Z^{2p}(\Omega_{X/\CC}^{\bullet \geq p}) \bigr).
\]
In this way, we obtain a well-defined class in $\HH_Z^{2p}(\Omega_{X/\CC}^{\bullet
\geq p})$, and hence in the algebraic de Rham cohomology
$\HH^{2p}(\Omega_{X/\CC}^{\bullet \geq p}) = F^p \HdR{2p}(X/\CC)$.

For the general case, one uses the theory of Chern classes, which associates to a
locally free sheaf $\shE$ of rank $r$ a sequence of Chern classes $c_1(\shE), \dotsc,
c_r(\shE)$. We recall their construction in Betti cohomology and in algebraic de Rham
cohomology.

First, consider the case of an algebraic line bundle $\shL$; as usual, we denote the
associated holomorphic line bundle by $\Lan$. The first Chern class $c_1(\Lan) \in
H^2 \bigl( \Xan, \ZZ(1) \bigr)$ can be defined using the exponential sequence
\[
0 \to \ZZ(1) \to \shO_{\Xan} \xrightarrow{\exp} \shO_{\Xan}^{\ast} \to 0.
\]
The isomorphism class of $\Lan$ belongs to $H^1(\Xan, \shO_{\Xan}^{\ast})$, and
$c_1(\Lan)$ is the image of this class under the connecting homomorphism.

To relate this to differential forms, cover $X$ by open subsets $U_i$ on which $\Lan$
is trivial, and let $g_{ij} \in \shO_{\Xan}^{\ast}(U_i \cap U_j)$ denote the
holomorphic transition functions for this cover. If each $U_i$ is simply connected,
say, then we can write $g_{ij} = e^{f_{ij}}$, and then 
\[
	f_{jk} - f_{ik} + f_{ij} \in \ZZ(1)
\]
form a $2$-cocycle that represents $c_1(\Lan)$. Its image in $H^2(\Xan, \CC) \simeq
\HH^2(\Omega_{\Xan}^{\bullet})$ is cohomologous to the class of the $1$-cocycle
$df_{ij}$ in $H^1(\Xan, \Omega_{\Xan}^1)$. But $df_{ij} = dg_{ij} / g_{ij}$, and so
$c_1(\Lan)$ is also represented by the cocycle $dg_{ij} / g_{ij}$. This explains
the special case $p=1$ in Bloch's construction.

To define the first Chern class of $\shL$ in algebraic de Rham cohomology, we use the
fact that a line bundle is also locally trivial in the Zariski topology. If $U_i$ are
Zariski-open sets on which $\shL$ is trivial, and $g_{ij} \in \shO_X^{\ast}(U_i \cap
U_j)$ denotes the corresponding transition functions, we can define $\cdR_1(\shL) \in
F^1 \HdR{2}(X/\CC)$ as the hypercohomology class determined by the cocycle $dg_{ij} / g_{ij}$.
In conclusion, we then have $\cdR_1(\shL) = c_1(\Lan)$ under the isomorphism in
Grothendieck's theorem.

Now suppose that $\shE$ is a locally free sheaf of rank $r$ on $X$. On the associated
projective bundle $\pi \colon \PP(\shE) \to X$, we have a universal line bundle
$\shO_{\shE}(1)$, together with a surjection from $\piu \shE$. In Betti cohomology,
we have
\[
	H^{2r} \bigl( \PP(\Ean), \ZZ(r) \bigr) =
		\bigoplus_{i=0}^{r-1} \xi^i \cdot \piu H^{2r-2i} \big( \Xan, \ZZ(r-i) \bigr),
\]
where $\xi \in H^2 \bigl( \PP(\Ean), \ZZ(1) \bigr)$ denotes the first
Chern class of $\shO_{\shE}(1)$. Consequently, there are unique classes $c_k \in
H^{2k} \bigl( \Xan, \ZZ(k) \bigr)$ that satisfy the relation
\[
	\xi^r - \piu(c_1) \cdot \xi^{r-1} + \piu(c_2) \cdot \xi^{r-2} + \dotsb + (-1)^r \piu(c_r) = 0,
\]
and the $k$-th Chern class of $\Ean$ is defined to be $c_k(\Ean) = c_k$.
The same construction can be carried out in algebraic de
Rham cohomology, producing Chern classes $\cdR_k(\shE) \in F^k H^{2k}(X/\CC)$.
It follows easily from the case of line bundles that we have
\[
	\cdR_k(\shE) = c_k(\Ean)
\]
under the isomorphism in Grothendieck's theorem.

Since coherent sheaves admit finite resolutions by locally free sheaves, it is
possible to define Chern classes for arbitrary coherent sheaves. One consequence of
the Riemann-Roch theorem is the equality
\[
	\class{\Zan} = \frac{(-1)^{p-1}}{(p-1)!} c_p(\shO_{\Zan}) 
		\in H^{2p} \bigl( \Xan, \QQ(p) \bigr).
\]
Thus it makes sense to define the cycle class of $Z$ in algebraic de Rham cohomology
by the formula
\[
	\classdR{Z} = \frac{(-1)^{p-1}}{(p-1)!} c_p(\shO_Z) \in F^p \HdR{2p}(X/\CC).
\]
It follows that $\classdR{Z} = \class{\Zan}$, and so the cycle class of $\Zan$ can
indeed be constructed algebraically, as claimed.

\begin{exc}\label{exc}
Let $X$ be a nonsingular projective variety defined over $\CC$, let $D \subseteq X$ be a
nonsingular hypersurface, and set $U = X - D$. One can show that $\HdR{i}(U/\CC)$ is
isomorphic to the hypercohomology of the log complex $\Omega_{X/\CC}^{\bullet}(\log
D)$. Use this to construct a long exact sequence
\[
\dotsb \to \HdR{i-2}(D) \to \HdR{i}(X) \to \HdR{i}(U) \to \HdR{i-1}(D) \to \dotsb
\]
for the algebraic de Rham cohomology groups. Conclude by induction on the dimension
of $X$ that the restriction map
\[
	\HdR{i}(X/\CC) \to \HdR{i}(U/\CC)
\]
is injective for $i \leq 2 \codim Z - 1$, and an isomorphism for $i \leq 2 \codim Z -
2$.
\end{exc}


\section{Absolute Hodge classes}

 In this section, we introduce the notion of absolute Hodge classes in the cohomology of a complex algebraic variety.
While Hodge theory applies to general compact K\"ahler manifolds, absolute Hodge classes are brought in as a way
to deal with cohomological properties of a variety coming from its algebraic structure.

This circle of ideas is closely connected to the motivic philosophy as envisioned by Grothendieck. One of
the goals of this text is to hopefully give a hint of how absolute Hodge classes can allow one to give unconditional
proofs for results of a motivic flavor.

\subsection{Algebraic cycles and the Hodge conjecture}\label{Hodge-cj}

As an example of the need for a suitable structure on the cohomology of a complex algebraic variety that uses more than
usual Hodge theory, let us first discuss some aspects of the Hodge conjecture.

\bigskip

Let $X$ be a smooth projective variety over $\C$. The singular cohomology groups of $X$ are endowed with pure Hodge
structures such that for any integer $p$, $H^{2p}(X, \Z(p))$ has weight $0$. We denote by $Hdg^p(X)$ the group of Hodge
classes in $H^{2p}(X, \Z(p))$.

As we showed earlier, if $Z$ is a subvariety of $X$ of codimension $p$, its cohomology class $[Z]$ in $H^{2p}(X, \Z(p))$
is a Hodge class. The Hodge conjecture states that the cohomology classes of subvarieties of $X$ span the
$\Q$-vector space generated by Hodge classes.

\begin{conj}
Let $X$ be a smooth projective variety over $\C$. For any nonnegative integer $p$, the subspace of degree $p$ rational
Hodge classes
$$Hdg^p(X)\otimes \Q\subset H^{2p}(X, \Q(p))$$
is generated over $\Q$ by the cohomology classes of codimension $p$ subvarieties of $X$.
\end{conj}

\bigskip

If $X$ is only assumed to be a compact K\"ahler manifold, the cohomology groups $H^{2p}(X, \Z(p))$ still carry Hodge
structures, and analytic subvarieties of $X$ still give rise to Hodge classes. While a general compact K\"ahler manifold
can have very few analytic subvarieties, Chern classes of coherent sheaves also are Hodge classes on the cohomology of
$X$.

Note that on a smooth projective complex variety, analytic subvarieties are algebraic by the GAGA principle of Serre,
and that Chern classes of coherent sheaves are linear combinations of cohomology classes of algebraic subvarieties of
$X$. Indeed, this is true for locally free sheaves and coherent sheaves on a smooth variety have finite free
resolutions. This latter result is no longer true for general compact K\"ahler manifolds, and indeed Chern classes of
coherent sheaves can generate a strictly larger subspace than that generated by the cohomology classes of analytic
subvarieties.

These remarks show that the Hodge conjecture could be generalized to the K\"ahler setting by asking whether Chern
classes
of coherent sheaves on a compact K\"ahler manifold generate the space of Hodge classes. This would be the natural
Hodge-theoretic framework for this question. However, the answer to this question is negative, as proved by Voisin in
\cite{Vo02}.

\begin{thm}
There exists a compact K\"ahler manifold $X$ such that $Hdg^2(X)$ is nontorsion while for any coherent sheaf $\mathcal
F$ on $X$, $c_2(\mathcal F)=0$, $c_2(\mathcal F)$ denoting the second Chern class of $\mathcal F$.
\end{thm}

The proof of the preceding theorem takes $X$ to be a general Weil torus. Weil tori are complex tori with a specific
linear algebra condition which endows them with a nonzero space of Hodge classes. While here they provide a
counterexample to the Hodge conjecture in the K\"ahler setting, they will be instrumental, in the projective case, to
proving Deligne's theorem on absolute Hodge classes.

To our knowledge, there is no tentative formulation of a Hodge conjecture for compact K\"ahler manifolds. As a
consequence, one has to bring ingredients from the algebraic world to tackle the Hodge conjecture for projective
varieties.

\subsection{Galois action, algebraic de Rham cohomology and absolute Hodge
classes}\label{definitions}

The preceding paragraph suggests that the cohomology of projective complex varieties has a richer underlying structure
than that of a general K\"ahler manifold.

This brings us very close to the theory of motives, which Grothendieck created in the sixties as a way to encompass
cohomological properties of algebraic varieties. Even though these notes won't use the language of motives, the motivic
philosophy is pervasive to all the results we will state.

\bigskip

Historically, absolute Hodge classes were introduced by Deligne in order to make an unconditional use of motivic ideas.
We will review his results in the next sections. The main starting point is, as we showed earlier, that the singular
cohomology of a smooth proper complex algebraic variety with complex coefficients can be computed algebraically, using
algebraic de Rham cohomology.

Indeed, let $X$ be a smooth proper complex algebraic variety defined over $\C$. We have a canonical isomorphism
$$H^*(X^{an}, \C)\simeq \mathbb{H}^*(\Omega^{\bullet}_{X/\C}),$$
where $\Omega^{\bullet}_{X/\C}$ is the algebraic de Rham complex of the variety $X$ over $\C$. A striking consequence of
this isomorphism is that the singular cohomology of the manifold $X^{an}$ with complex definition can be computed
algebraically.  It follows that the topology of $\C$ is not needed in this computation, the field
structure being sufficient. More generally, if $X$ is a smooth proper variety defined over any field $k$ of
characteristic zero, the hypercohomology of the de Rham complex of $X$ over Spec$\, k$ gives a $k$-algebra which by
definition is the algebraic de Rham cohomology of $X$ over $k$.

\bigskip

Now let $Z$ be an algebraic cycle of codimension $p$ in $X$. As we showed earlier, $Z$ has a cohomology class
$$[Z]\in H^{2p}(X^{an}, \Q(p))$$
which is a Hodge class, that is, the image of $[Z]$ in $H^{2p}(X^{an}, \C(p))\simeq H^{2p}(X/ \C)(p)$ lies in
$$F^0 H^{2p}(X/ \C)(p)=F^p H^{2p}(X/ \C).$$

\bigskip

Given any automorphism $\sigma$ of the field $\C$, we can form the conjugate variety $X^{\sigma}$ defined as the
complex variety $X\times_{\sigma} \mathrm{Spec}\,\C$, that is, by the cartesian diagram
\begin{equation}
 \xymatrix{
X^{\sigma}\ar[r]^{\sigma^{-1}}\ar[d] & X\ar[d]\\
\mathrm{Spec}\,\C\ar[r]^{\sigma^*} & \mathrm{Spec}\,\C.
}
\end{equation}
It is a smooth projective variety. If $X$ is defined by homogeneous polynomials $P_1, \ldots, P_r$ in some projective
space, then $X^{\sigma}$ is defined by the conjugates of the $P_i$ by $\sigma$. In this case, the morphism from
$X^{\sigma}$ to $X$ in the cartesian diagram sends the closed point with coordinates $(x_0 : \ldots : x_n)$ to the
closed point with homogeneous coordinates $(\sigma^{-1}(x_0):\ldots:\sigma^{-1}(x_n))$, which allows us to denote it by
$\sigma^{-1}$.

The morphism $\sigma^{-1} : X^{\sigma}\ra X$ is an isomorphism of abstract schemes, but it is not a morphism of complex
varieties.
Pull-back of K\"ahler forms still induces an isomorphism between the de Rham complexes of $X$ and $X^{\sigma}$
\begin{equation}\label{iso-cx}
 (\sigma^{-1})^* \Omega^{\bullet}_{X/\C} \stackrel{\sim}{\ra} \Omega^{\bullet}_{X^{\sigma}/\C}.
\end{equation}

Taking hypercohomology, we get an isomorphism
$$(\sigma^{-1})^* : H^*(X/\C) \stackrel{\sim}{\ra} H^*(X^{\sigma}/\C), \alpha\mapsto \alpha^{\sigma}.$$
Note however that this isomorphism is not $\C$-linear, but $\sigma$-linear, that is, if $\lambda\in\C$, we have
$(\lambda\alpha)^{\sigma}=\sigma(\lambda)\alpha^{\sigma}$. We thus get an isomorphism of complex vector
spaces
\begin{equation}\label{iso-dR}
H^*(X/\C)\otimes_{\sigma}\C \stackrel{\sim}{\ra} H^*(X^{\sigma}/\C)
\end{equation}
between the de Rham cohomology of $X$ and that of $X^{\sigma}$. Here the notation $\otimes_{\sigma}$ means that we are
taking tensor product with $\C$ mapping to $\C$ via the morphism $\sigma$. Since this isomorphism comes from an
isomorphism of the de Rham complexes, it preserves the Hodge filtration.

\bigskip

The preceding construction is compatible with the cycle map. Indeed, $Z$ being as before a codimension $p$ cycle in
$X$, we can form its conjugate $Z^{\sigma}$ by $\sigma$. It is a codimension $p$ cycle in $X^{\sigma}$. The
construction of the cycle class map in de Rham cohomology shows that we have
$$[Z^{\sigma}]=[Z]^{\sigma}$$
in $H^{2p}(X^{\sigma}/\C)(p)$. It lies in $F^0 H^{2p}(X^{\sigma}/ \C)(p)$.

Now as before $X^{\sigma}$ is a smooth projective complex variety, and its de Rham cohomology group $H^{2p}(X^{\sigma}/
\C)(p)$ is canonically isomorphic to the singular cohomology group $H^{2p}((X^{\sigma})^{an}, \C(p)).$ The cohomology
class $[Z^{\sigma}]$ in $H^{2p}((X^{\sigma})^{an}, \C(p))\simeq H^{2p}(X^{\sigma}/\C)(p)$ is a Hodge class. This leads
to the following definition.

\begin{df}
 Let $X$ be a smooth complex projective variety. Let $p$ be a nonnegative integer, and let $\alpha$ be an element of
$H^{2p}(X/\C)(p)$. The cohomology class $\alpha$ is an absolute Hodge class if for every automorphism $\sigma$ of $\C$,
the cohomology class $\alpha^{\sigma}\in H^{2p}((X^{\sigma})^{an}, \C(p))\simeq H^{2p}(X^{\sigma}/\C(p))$ is a Hodge
class\footnote{Since $H^{2p}((X^{\sigma})^{an}, \C)$ is only considered as a vector space here, the Tate twist might
seem superfluous. We put it here to emphasize that the comparison isomorphism with de Rham cohomology contains a factor
$(2 \pi i)^{-p}$.}.
\end{df}

The preceding discussion shows that the cohomology class of an algebraic cycle is an absolute Hodge class. Taking
$\sigma=\mathrm{Id}_{\C}$, we see that absolute Hodge classes are Hodge classes.

Using the canonical isomorphism $H^{2p}(X^{an}, \C(p))\simeq H^{2p}(X/ \C)(p)$, we will say that a class in
$H^{2p}(X^{an}, \C)$ is absolute Hodge if its image in $H^{2p}(X/ \C)(p)$ is.

\bigskip

We can rephrase the definition of absolute Hodge cycles in a slightly more intrinsic way. Let $k$ be a field of
characteristic zero, and let $X$ be a smooth projective variety defined over $k$. Assume that the cardinality of $k$ is
less or equal to the cardinality of $\C$, so that there exist embeddings of $k$ into $\C$. Note that any variety
defined over a field of characteristic zero is defined over such a field, as it is defined over a field generated over
$\Q$ by a finite number of elements.

\begin{df}\label{def-dR}
Let $p$ be a nonnegative integer, and let $\alpha$ be an element of the de Rham cohomology space $H^{2p}(X/k)$. Let
$\tau$ be an embedding of $k$ into $\C$, and let $\tau X$ be the complex variety obtained from $X$ by base change to
$\C$. We say that $\alpha$ is a Hodge class relative to $\tau$ if the image of $\alpha$ in
$$ H^{2p}(\tau X/\C)=H^{2p}(X/k)\otimes_{\tau}\C$$
is a Hodge class. We say that $\alpha$ is absolute Hodge if it is a Hodge class relative to every embedding of $k$ into
$\C$.
\end{df}

Let $\tau$ be any embedding of $k$ into $\C$. Since by standard field theory, any two embeddings of $k$ into $\C$ are
conjugated by an automorphism of $\C$, it is straightforward to check that such a cohomology class $\alpha$ is absolute
Hodge if and only if its image in $H^{2p}(\tau X/\C)$ is. Definition \ref{def-dR} has the advantage of not referring to
automorphisms of $\C$, as it is impossible to exhibit one except for the identity and complex conjugation.

This definition allows us to work with absolute Hodge classes in a wider setting by using other cohomology theories.
Even though we will be mainly concerned with de Rham cohomology in these
notes, let us state what absolute Hodge classes are for \'etale cohomology.

\begin{df}\label{def-etale}
 Let $\overline{k}$ be an algebraic closure of $k$. Let $p$ be a nonnegative integer, $l$ a prime number, and let
$\alpha$ be an element of the \'etale cohomology space $H^{2p}(X_{\overline{k}}, \Q_l(p))$. Let
$\tau$ be an embedding of $\overline{k}$ into $\C$, and let $\tau X$ be the complex variety obtained from
$X_{\overline{k}}$ by base change to $\C$. We say that $\alpha$ is a Hodge class relative to $\tau$ if the image of
$\alpha$ in
$$ H^{2p}((\tau X)^{an}, \Q_l(p))\simeq H^{2p}(X_{\overline{k}}, \Q_l(p))$$
is a Hodge class, that is, if it lies in the rational subspace $ H^{2p}((\tau X)^{an}, \Q(p))$ of $ H^{2p}((\tau
X)^{an}, \Q_l(p))$ and is a Hodge class. We say that $\alpha$ is absolute Hodge if it is a Hodge class relative to every
embedding of $k$ into $\C$.
\end{df}

\begin{rk}
 It is not clear whether an absolute Hodge class in the sense of definition \ref{def-dR} is always the first component
of an absolute Hodge class in the sense of definition \ref{def-etale}, see \cite{Del82}, Question 2.4.
\end{rk}

\begin{rk}
 It is possible to encompass crystalline cohomology in a similar framework, see for instance \cite{Bl}.
\end{rk}

\begin{rk}\label{mixed}
It is possible to work with absolute Hodge classes on more general varieties. Indeed, while the definitions we gave
above only deal with the smooth projective case, the fact that the singular cohomology of any quasi-projective
variety can be computed using suitable versions of algebraic de Rham cohomology -- whether through logarithmic de Rham
cohomology, algebraic de Rham cohomology on simplicial schemes or a combination of the two -- makes it possible to
consider absolute Hodge classes in the singular cohomology groups of a general complex variety. 

Note here that if $H$ is a mixed Hodge structure defined over $\Z$ with weight filtration $W_{\bullet}$ and Hodge
filtration $F^{\bullet}$, a Hodge class in $H$ is an element of $H_{\Z}\bigcap F^0 H_{\C} \bigcap W_0 H_{\C}$. One of
the main specific features of absolute Hodge classes on quasi-projective varieties is that they can be found in the odd
singular cohomology groups. Let us consider the one-dimensional case as an example. Let $C$ be a smooth complex
projective curve, and let $D=\Sigma_i n_i P_i$ be a divisor of degree $0$ on $C$. Let $Z$ be the support of $D$, and let
$C'$ be the complement of $Z$ in $C$. It is a smooth quasi-projective curve.

As in Exercise \ref{exc}, we have an exact sequence
$$0\ra H^1(C, \Q(1))\ra H^1(C', \Q(1)) \ra H^0(Z, \Q) \ra H^2(C, \Q(1)).$$
The divisor $D$ has a cohomology class $d\in H^0(Z, \Q)$. Since the degree of $D$ is zero, $d$ maps to zero in $H^2(C,
\Q(1))$. As a consequence, it comes from an element in $H^1(C', \Q(1))$. Now it can be proved that there exists a Hodge
class in $H^1(C', \Q(1))$ mapping to $d$ if and only if some multiple of the divisor $D$ is rationally equivalent to
zero.

In general, the existence of Hodge classes in extensions of mixed Hodge structures is to be related to Griffiths'
Abel-Jacobi map, see \cite{Car}. The problem of whether those are absolute Hodge classes is linked with problems
pertaining to the Bloch-Beilinson filtration and comparison results with regulators in \'etale cohomology, see
\cite{Jan}. 

While we will not discuss here specific features of this problem, most of the results we will state in the pure case
have extensions to the mixed case, see for instance \cite{ChAJ}.
\end{rk}

\subsection{Variations on the definition and some functoriality properties}\label{variations}

While the goal of these notes is neither to construct nor to discuss the category motives one can obtain using
absolute Hodge classes, we will need to use functoriality properties of absolute Hodge classes that are very close to
those motivic constructions. In this paragraph, we extend the definition of absolute Hodge classes to encompass
morphisms, multilinear forms, etc. This almost amounts to defining motives for absolute Hodge classes as in \cite{DMOS}.
The next paragraph will be devoted to semi-simplicity results through the use of polarized Hodge structures.

\bigskip

The following generalizes Definition \ref{def-dR}.
\begin{df}
Let $k$ be a field of characteristic zero with cardinality less or equal than the cardinality of $\C$. Let $(X_i)_{i\in
I}$ and $(X_j)_{j\in J}$ be smooth projective varieties over $\C$, and let $(p_i)_{i\in I}$, $(q_j)_{j\in J}$, $n$ be
integers.Let $\alpha$ be an element of the tensor product
$$(\bigotimes_{i\in I} H^{p_i}(X_i/k))\otimes (\bigotimes_{j\in J} H^{q_j}(X_j/k)^*)(n).$$

Let $\tau$ be an embedding of $k$ into $\C$. We say that $\alpha$ is a Hodge class relative to $\tau$ if the image of
$\alpha$ in
$$(\bigotimes_{i\in I} H^{p_i}(X_i/k))\otimes (\bigotimes_{j\in J}
H^{q_j}(X_j/k)^*)(n)\otimes_{\tau}\C=(\bigotimes_{i\in I} H^{p_i}(\tau X_i/\C))\otimes (\bigotimes_{j\in J} H^{q_j}(\tau
X_j/\C)^*)(n)$$
is a Hodge class. We say that $\alpha$ is absolute Hodge if it is a Hodge class relative to every embedding of $k$ into
$\C$.
\end{df}
As before, if $k=\C$, we can speak of absolute Hodge classes in the group
$$(\bigotimes_{i\in I} H^{p_i}(X_i, \Q))\otimes (\bigotimes_{j\in J} H^{q_j}(X_j, \Q)^*)(n).$$
If $X$ and $Y$ are two smooth projective complex varieties, and if
$$f : H^p(X, \Q(i))\ra H^q(Y, \Q(j))$$
is a morphism of Hodge structures, we will say that $f$ is absolute Hodge, or is given by an absolute Hodge class, if
the element corresponding to $f$ in
$$H^q(Y, \Q)\otimes H^p(X, \Q)^*(j-i)$$
is an absolute Hodge class. Similarly, we can define what it means for a multilinear form, e.g., a polarization, to be
absolute Hodge.

This definition allows us to exhibit elementary examples of Hodge classes as follows.
\begin{ex}
Let $X$ be a smooth projective complex variety.
\begin{itemize}
\item Cup-product defines a map
$$H^p(X, \Q)\otimes H^q(X, \Q)\ra H^{p+q}(X, \Q).$$
This map is given by an absolute Hodge class.
\item Poincar\'e duality defines an isomorphism
$$H^p(X, \Q)\ra H^{2d-p}(X, \Q(d))^*,$$
where $d$ is the dimension of $X$. This map is given by an absolute Hodge class.
\end{itemize}
\end{ex}

\begin{proof}
This is formal. Let us write down the computations involved. Assume $X$ is defined over $k$ (which might be $\C$). We
have a cup-product map
$$H^p(X/k)\otimes H^q(X/k)\ra H^{p+q}(X/k).$$
Let $\tau$ be an embedding of $k$ into $\C$. The induced map
$$H^p(\tau X/\C)\otimes H^q(\tau X/\C)\ra H^{p+q}(\tau X/\C)$$
is cup-product on the de Rham cohomology of $\tau X$. We know that cup-product on a smooth complex projective variety is
compatible with Hodge structures, which shows that it is given by a Hodge class. The conclusion follows, and a very
similar argument proves the result regarding Poincar\'e duality.
\end{proof}

\bigskip

Morphisms given by absolute Hodge classes behave in a functorial way. The following properties are easy to prove,
working as in the preceding example to track down compatibilities.

\begin{prop}\label{functor}
Let $X$, $Y$ and $Z$ be smooth projective complex varieties, and let
$$f : H^p(X, \Q(i))\ra H^q(Y, \Q(j)),\, g : H^q(Y, \Q(j))\ra H^r(Y, \Q(k))$$
be morphisms of Hodge structures.
\begin{enumerate}
\item If $f$ is induced by an algebraic correspondence, then $f$ is absolute Hodge.
\item If $f$ and $g$ are absolute Hodge, then $g\circ f$ is absolute Hodge.
\item Let
$$f^{\dagger}: H^{2d'-q}(Y, \Q(d'-j))\ra H^{2d-p}(X, \Q(d-i))$$
be the adjoint of $f$ with respect to Poincar\'e duality. Then $f$ is absolute Hodge if and only if $f^{\dagger}$ is
absolute Hodge.
\item If $f$ is an isomorphism, then $f$ is absolute Hodge if and only if $f^{-1}$ is absolute Hodge.
\end{enumerate}
\end{prop}
Note that the last property is not known to be true for algebraic correspondences. For those, it is equivalent to the
Lefschetz standard conjecture, see the next paragraph. We will need a refinement of this property as follows.

\begin{prop}\label{inverse}
Let $X$ and $Y$ be smooth projective complex varieties, and let
$$p : H^p(X, \Q(i))\ra H^p(X, \Q(i))\, \mathrm{and}\, q : H^q(Y, \Q(j))\ra  H^q(Y, \Q(j))$$
be projectors. Assume that $p$ and $q$ are absolute Hodge. Let $V$ (resp. $W$) be the image of $p$ (resp. $q$), and let
$$ f : H^p(X, \Q(i))\ra H^q(Y, \Q(j))$$
be absolute Hodge. Assume that $qfp$ induces an isomorphism from $V$ to $W$. Then the composition
$$\xymatrix{H^q(Y, \Q(j)) \ar@{>>}[r] & W\ar[r]^{(qfp)^{-1}} & V \ar@{^{(}->}[r] & H^p(X, \Q(i))}$$
is absolute Hodge.
\end{prop}

\begin{proof}
We need to check that after conjugating by any automorphism of $\C$, the above composition is given by a Hodge class.
Since $q$, $f$ and $p$ are absolute Hodge, we only have to check that this is true for the identity automorphism, which
is the case.
\end{proof}

This is to compare with Grothendieck's construction of the category of pure motives as a pseudo-abelian category, see
for instance \cite{An-book}.

\subsection{Classes coming from the standard conjectures and polarizations}\label{standard}

Let $X$ be a smooth projective complex variety of dimension $d$. The cohomology of $X\times X$ carries a number of Hodge
classes which are not known to be algebraic. The standard conjectures, as stated in \cite{Gr68}, predict that the
K\"unneth components of the diagonal and the inverse of the Lefschetz isomorphism are algebraic. A proof of those would
have a lot of consequences in the theory of pure motives. Let us prove that they are absolute Hodge classes. More
generally -- and informally -- any cohomology class obtained from absolute Hodge classes by canonical rational
constructions can be proved to be absolute Hodge.

\bigskip

First, let $\Delta$ be the diagonal of $X\times X$. It is an algebraic cycle of codimension $d$ in $X\times X$, hence it
has a cohomology class $[\Delta]$ in $H^{2d}(X\times X, \Q(d))$. By the K\"unneth formula, we have a canonical
isomorphism
of Hodge structures
$$H^{2d}(X\times X, \Q)\simeq \bigoplus_{i=0}^{2d} H^i(X, \Q)\otimes H^{2d-i}(X, \Q),$$
hence projections $H^{2d}(X\times X, \Q)\ra H^i(X, \Q)\otimes H^{2d-i}(X, \Q)$.
Let $\pi_i$ be the component of $[\Delta]$ in $H^i(X, \Q)\otimes H^{2d-i}(X, \Q)(d)\subset H^{2d}(X\times X, \Q)(d)$.
The cohomology classes $\pi_i$ are the called the K\"unneth components of the diagonal.

\begin{prop}
The K\"unneth components of the diagonal are absolute Hodge cycles.
\end{prop}

\begin{proof}
Clearly the $\pi_i$ are Hodge classes. Let $\sigma$ be an automorphism of $\C$. Denote by $\Delta^{\sigma}$ the diagonal
of $X^{\sigma}\times X^{\sigma}=(X\times X)^{\sigma}$, and by $\pi_i^{\sigma}$ the K\"unneth components of
$\Delta^{\sigma}$. Those are also Hodge classes.

Let $\pi_{i,dR}$ (resp. $(\pi_i^{\sigma})_{dR}$) denote the images of the $\pi_i$ (resp. $\pi_i^{\sigma}$) in the de
Rham cohomology of $X\times X$ (resp. $X^{\sigma}\times X^{\sigma})$. The K\"unneth formula holds for de Rham cohomology
and is compatible with the comparison isomorphism between de Rham and singular cohomology. It follows that
$$(\pi_i^{\sigma})_{dR}=(\pi_{i,dR})^{\sigma}.$$
Since the $(\pi_i^{\sigma})_{dR}$ are Hodge classes, the conjugates of $\pi_{i,dR}$ are, which concludes the proof.
\end{proof}

Fix an embedding of $X$ into a projective space, and let $h\in H^2(X, \Q(2))$ be the cohomology class of a hyperplane
section. The hard Lefschetz theorem states that for all $i\leq d$, the morphism
$$L^{d-i}=\cup h^{d-i} : H^i(X, \Q)\ra H^{2d-i}(X, \Q(d-i)), x\mapsto x\cup\xi^{d-i}$$
is an isomorphism.

\begin{prop}
The inverse $f_i : H^{2d-i}(X, \Q(d-i))\ra H^{i}(X, \Q) $ of the Lefschetz isomorphism is absolute Hodge.
\end{prop}

\begin{proof}
This an immediate consequence of Proposition \ref{functor}.
\end{proof}

As an immediate corollary, we get the following result.

\begin{cor}
Let $i$ be an integer such that $2i\leq d$. An element $x\in H^{2i}(X, \Q)$ is an absolute Hodge class if and only if
$x\cup \xi^{d-2i}\in H^{2d-2i}(X, \Q(d-2i))$ is an absolute Hodge class.
\end{cor}

\bigskip

Using the preceding results, one introduce polarized Hodge structures in the setting of absolute Hodge classes. Let us
start with an easy lemma.

\begin{lem}\label{Lef-decomp}
Let $X$ be a smooth projective complex variety of dimension $d$, and let $h\in H^2(X, \Q(1))$ be the cohomology class
of a hyperplane
section. Let $L$ denote the operator given by cup-product with $\xi$. Let $i$ be an integer. In the Lefschetz
decomposition
$$H^i(X, \Q)=\bigoplus_{j\geq 0} L^j H^{i-2j}(X, \Q)_{prim}$$
of the cohomology of $X$ into primitive parts, the projection of $H^i(X, \Q)$ onto $L^j H^{i-2j}(X, \Q)_{prim}
\hookrightarrow H^i(X, \Q)$ is given by an absolute Hodge class.
\end{lem}

\begin{proof}
By induction, it is enough to prove that the projection of $H^i(X, \Q)$ onto  $L H^{i-2}(X, \Q)$ is given by an
absolute Hodge class. While this could be proved by an argument of Galois equivariance as before, consider the
composition
$$L\circ f_i\circ L^{d-i+1} : H^i(X, \Q) \ra H^i(X, \Q)$$
where $f_i : H^{2d-i}(X, \Q)\ra H^{i}(X, \Q)$ is the inverse of the Lefschetz operator. It is the desired projection
since $H^k(S, \Q)_{prim}$ is the kernel of $L^{d-i+1}$ in $H^i(S, \Q)$ .
\end{proof}

This allows for the following result, which shows that the Hodge structures on the cohomology of smooth projective
varieties can be polarized by absolute Hodge classes.

\begin{prop}\label{polarization}
Let $X$ be a smooth projective complex variety and $k$ be an integer. There exists a pairing
$$Q : H^{k}(X, \Q)\otimes H^{k}(X, \Q)\ra \Q(-k)$$
which is given by an absolute Hodge class and turns $H^{k}(X, \Q)$ into a polarized Hodge structure.
\end{prop}

\begin{proof}
Let $d$ be the dimension of $X$. By the hard Lefschetz theorem, we can assume $k\leq d$. Let $H$ be an ample line bundle
on
$X$ with first Chern class $h\in H^2(X, \Q(1)$, and let $L$ be the endomorphism
of the cohomology of $X$ given by the cup-product with $h$. Consider the Lefschetz decomposition
$$H^{k}(X, \Q)=\bigoplus_{i\geq 0} L^i H^{k-2i}(X, \Q)_{prim}$$
of $H^{k}(X, \Q)$ into primitive parts. Let $s$ be the linear automorphism of
$H^{k}(X, \Q)$ which is given by multiplication by $(-1)^i$ on $L^i H^{k-2i}(X, \Q)_{prim}$.

By the Hodge index theorem, the pairing
$$H^{k}(X, \Q)\otimes H^{k}(X, \Q)\ra \Q(1), \, \alpha\otimes\beta \mapsto
\int_{X} \alpha \cup L^{d-k}(s(\beta))$$
turns $H^{2p}(\overline{\mathcal X}, \Q)$ into a polarized Hodge structure.

By Lemma \ref{Lef-decomp}, the projections of $H^{2p}(X, \Q)$ onto the factors $L^i
H^{2p-2i}(X, \Q)_{prim}$ are given by absolute Hodge classes. It follows that the morphism $s$ is
given by an absolute Hodge class.

Since cup-product is given by an absolute Hodge class, see \ref{variations}, and $L$ is induced by an algebraic
correspondence, it follows that the pairing $Q$ is given by an absolute Hodge class, which concludes the proof of the
proposition.
\end{proof}

\begin{prop}\label{semi-simple}
Let $X$ and $Y$ be smooth projective complex varieties, and let
$$f : H^p(X, \Q(i))\ra H^q(Y, \Q(j))$$
be a morphism of Hodge structures. Fix polarizations on the cohomology groups of $X$ and $Y$ given by absolute Hodge
classes. Then the orthogonal projection of $H^{p}(X, \Q(i))$ onto $\mathrm{Ker}\,f$ and
 the orthogonal projection of $H^q(Y, \Q(j))$ onto $\mathrm{Im}\,f)$ are given by absolute Hodge classes.
\end{prop}

\begin{proof}
The proof of this result is a formal consequence of the existence of polarizations by absolute Hodge classes. It is easy
to prove that the projections we consider are absolute using an argument of Galois equivariance as in the preceding
paragraph. Let us however give an alternate proof from linear algebra. The abstract argument corresponding to this proof
can be found in \cite{An}, section 3. We will only prove that the orthogonal projection of $H^{p}(X, \Q(i))$ onto
$\mathrm{Ker}\,f$ is absolute Hodge, the other statement being a consequence via Poincar\'e duality.

\bigskip

For ease of notation, we will not write down Tate twists. Those can be recovered by weight considerations. By Poincar\'e
duality, the polarization on
$H^{p}(X, \Q)$ induces an isomorphism
$$\phi : H^{p}(X, \Q) \ra H^{2d-p}(X, \Q),$$
where $d$ is the dimension of $X$, which is absolute Hodge since the polarization is. Similarly, the polarization on
$H^{q}(Y, \Q)$ induces a morphism
$$\psi: H^{q}(Y, \Q) \ra H^{2d'-q}(Y, \Q)$$
where $d'$ is the dimension of $Y$, which is given by an absolute Hodge class.

Consider the following diagram, which does not commute
$$\xymatrix{H^{p}(X,\Q)\ar[d]^{f}\ar[r]^{\phi} & H^{2d-p}(X, \Q)\\
            H^{q}(Y, \Q)\ar[r]^{\psi} & H^{2d'-q}(Y,\Q)\ar[u]^{f^{\dagger}}},$$
and consider the morphism
$$h : H^{p}(X, \Q)\ra H^{p}(X, \Q), x\mapsto (\phi^{-1}\circ f^{\dagger}\circ \psi
\circ f)(x).$$
Since all the morphisms in the diagram above are given by absolute Hodge classes, $h$ is. Let us compute the kernel
and the image of $h$.

Let $x\in  H^{p}(X, \Q)$. We have $h(x)=0$ if and only if $f^{\dagger}\psi f(x)=0$, which means that
for all $y$ in $H^{p}(X, \Q)$,
$$f^{\dagger}\psi f(x)\cup y=0,$$
that is, since $f$ and $f^{\dagger}$ are transpose of each other :
$$\psi f(x)\cup f(y)=0,$$
which exactly means that $f(x)$ is orthogonal to $f(H^{p}(X, \Q))$ with respect to the
polarization of $H^{q}(Y, \Q)$. Now the space $f(H^{p}(X, \Q))$ is a Hodge
substructure of the polarized Hodge structure $H^{q}(Y, \Q)$. As such, it does not contain any nonzero
totally isotropic element. This implies that $f(x)=0$ and shows that
$$\mathrm{Ker}\,h=\mathrm{Ker}\, f.$$

Since $f$ and $f^{\dagger}$ are transpose of each other, the image of $h$ is clearly contained
in $(\mathrm{Ker}\,f)^{\bot}$. Considering the rank of $h$, this readily shows that
$$\mathrm{Im}\,h =(\mathrm{Ker}\,f)^{\bot}.$$

The two subspaces $\mathrm\mathrm{Ker}\,h={Ker}\, f$ and $\mathrm{Im}\,h=(\mathrm{Ker}\,f)^{\bot}$ of
$H^{p}(X, \Q)$ are in direct sum. By standard linear algebra, it follows that the orthogonal
projection $p$ of $H^{p}(X, \Q)$ onto $(\mathrm{Ker}\,f)^{\bot}$ is a polynomial in $h$ with
rational coefficients. Since $h$ is given by an absolute Hodge class, so is $p$, as well as $\mathrm{Id}-p$, which
is the orthogonal projection onto $\mathrm{Ker}\,f$.
\end{proof}

\begin{cor}\label{antecedent}
 Let $X$ and $Y$ be two smooth projective complex varieties, and let
$$f : H^p(X, \Q(i))\ra H^q(Y, \Q(j))$$
be a morphism given by an absolute Hodge class. Let $\alpha$ be an absolute Hodge class in the image of $f$. Then there
exists an absolute Hodge class $\beta\in H^p(X, \Q(i))$ such that $f(\beta)=\alpha$.
\end{cor}

\begin{proof}
By Proposition \ref{semi-simple}, the orthogonal projection of $H^{p}(X, \Q(i))$ onto $(\mathrm{Ker}\,f)^{\bot}$ and
 the orthogonal projection of $H^q(Y, \Q(j))$ onto $\mathrm{Im}\,f)$ are given by absolute Hodge classes. Now
Proposition \ref{inverse} shows that the composition
$$\xymatrix{H^q(Y, \Q(j)) \ar@{>>}[r] & \mathrm{Im}\,f\ar[r]^{(qfp)^{-1}} & (\mathrm{Ker}\,f)^{\bot} \ar@{^{(}->}[r] &
H^p(X, \Q(i))}$$
is absolute Hodge. As such, it sends $\alpha$ to an absolute Hodge class $\beta$. Since $\alpha$ belongs to the image
of $f$, we have $f(\beta)=\alpha$.
\end{proof}

The results we proved in this paragraph and the preceding one are the ones needed to construct a category of motives
for absolute Hodge cycles and prove it is a semi-simple abelian category. This is done in \cite{DMOS}. In that sense,
absolute Hodge classes provide a way to work with an unconditional theory of motives, to quote Andr\'e.

We actually proved more. Indeed, while the explicit proofs we gave of Proposition \ref{Lef-decomp} and Proposition
\ref{semi-simple} might seem a little longer than what would be needed, they provide the cohomology classes we need
using only classes coming from the standard conjectures. This is the basis for Andr\'e's notion of motivated cycles
described in \cite{An}. This paper shows that a lot of the results we obtain here about the existence of some absolute
Hodge classes can be actually strengthened to motivated cycles. In particular, the algebraicity of the absolute Hodge
classes we consider, which is a consequence of the Hodge conjecture, is most of the time implied by the standard
conjectures.

\subsection{Absolute Hodge classes and the Hodge conjecture}

Let $X$ be a smooth projective complex variety. We proved earlier that the cohomology class of an algebraic cycle in
$X$ is absolute Hodge. This remark allows us to split the Hodge conjecture in the two following conjectures.

\begin{conj}\label{H-implies-AH}
 Let $X$ be a smooth projective complex variety. Let $p$ be a nonnegative integer, and let $\alpha$ be an element of
$H^{2p}(X, \Q(p))$. Then $\alpha$ is a Hodge class if and only if it is an absolute Hodge class.
\end{conj}

\begin{conj}\label{AH-implies-alg}
 Let $X$ be a smooth projective complex variety. For any nonnegative integer $p$, the subspace of degree $p$ absolute
Hodge classes is generated over $\Q$ by the cohomology classes of codimension $p$ subvarieties of $X$.
\end{conj}

Those statements are a tentative answer to the problem we raised in paragraph \ref{Hodge-cj}. Indeed, while those two
conjectures together imply the Hodge conjecture, neither of them makes sense in the setting of K\"ahler manifolds. The
main feature of an algebraic variety $X$ we use here, which is not shared by general compact K\"ahler manifolds, is that
such an $X$ always can be realized as the generic fiber of a morphism of algebraic varieties defined over a number
field. We will go back to this point of view in the next section when studying the relationship between Hodge loci and
absolute Hodge classes. The fact that we can use varieties over number fields is the reason why Galois actions can enter
the picture in the algebraic setting.

This situation however does not happen outside the algebraic world, as automorphisms of $\C$ other than the identity
and complex conjugation are very discontinuous -- e.g., they are not measurable. This makes it impossible to give a
meaning to the conjugate of a complex manifold by an automorphism of $\C$.

\bigskip

Even for algebraic varieties, the fact that automorphisms of $\C$ are highly discontinuous appears. Let $\sigma$ be an
automorphism of $\C$, and let $X$ be a smooth projective complex variety. Equation (\ref{iso-dR}) induces a
$\sigma$-linear isomorphism
$$(\sigma^{-1})^* : H^*(X^{an}, \C) \ra H^*((X^{\sigma})^{an}, \C)$$
between the singular cohomology with complex coefficients of the complex manifolds underlying $X$ and $X^{\sigma}$.
Conjecture \ref{H-implies-AH} means that Hodge classes in $H^*(X^{an}, \C)$ should map to Hodge classes in
$H^*((X^{\sigma})^{an}, \C)$. In particular, they should map to elements of the rational subspace
$H^*((X^{\sigma})^{an}, \Q)$.

However, it is not to be expected that $(\sigma^{-1})^*$ maps $H^*(X^{an}, \Q)$ to $H^*((X^{\sigma})^{an}, \Q)$.
It can even happen that the two algebras $H^*(X^{an}, \Q)$ and $H^*((X^{\sigma})^{an}, \Q)$ are not isomorphic, see
\cite{Ch}. This implies in particular that the complex varieties  $X^{an}$ and  $(X^{\sigma})^{an}$ need not be
homeomorphic, while the schemes $X$ and $X^{\sigma}$ are. This also shows that singular cohomology with rational
algebraic coefficients can not be defined algebraically\footnote{While the isomorphism we gave between the algebras
$H^*(X^{an}, \C)$ and $H^*((X^{\sigma})^{an}, \C)$ is not $\C$-linear, it is possible to show using \'etale cohomology
that there exists a $\C$-linear isomorphism between those two algebras, depending on an embedding of $\Q_l$ into
$\C$.}.

\bigskip

The main goal of these notes is to discuss Conjecture \ref{H-implies-AH}. We will give a number of example of absolute
Hodge classes which are not known to be algebraic, and describe some applications. While Conjecture
\ref{AH-implies-alg} seems to be very open at the time, we can make two remarks about it.

Let us first state a result which might stand as a motivation for the statement of this conjecture. We mentioned
above that conjugation by an automorphism of $\C$ does not in general preserve singular cohomology with rational
coefficients, but it does preserve absolute Hodge classes by definition.

Let $X$ be a smooth projective complex variety. The singular cohomology with rational coefficients of the underlying
complex manifold $X^{an}$ is spanned by the cohomology classes of images of real submanifolds of $X^{an}$. The next
result, see
\cite{votakagi}, Lemma 28 for a related statement, shows that among closed subsets of $X^{an}$, algebraic subvarieties
are the only one that remain closed after conjugation by an automorphism of $\C$.

Recall that if $\sigma$ is an automorphism of $\C$, we have an isomorphism of schemes
$$\sigma : X \ra X^{\sigma}.$$ It sends complex points of $X$ to complex points of $X^{\sigma}$.

\begin{prop}
Let $X$ be a complex variety, and let $F$ be a closed subset of $X^{an}$. Assume that for any automorphism $\sigma$
of $\C$, the subset
$$\sigma(F)\subset X^{\sigma}(\C)$$
is closed in $(X^{\sigma})^{an}$. Then $F$ is a countable union of algebraic subvarieties of $X$.
\end{prop}

\begin{proof}
Using induction on the dimension of $X$, we can assume that $F$ is not contained in a countable union of proper
subvarieties of $X$. We want to prove that $F=X$. Using a finite map from $X$ to a projective space, we can
assume that $X=\mathbb{A}^n_{\C}$.  Our hypothesis is thus that $F$ is a closed subset of $\C^n$ which is not contained
in a countable union of proper subvarieties of $\C^n$, such that for any automorphism of $\C$,
$\sigma(F)=\{(\sigma(x_1), \ldots, \sigma(x_n)), (x_1, \ldots, x_n)\in \C^n\}$
is closed in $\C^n$.  We will use an elementary lemma.

\begin{lem}
Let $k$ be a countable subfield of $\C$. There exists a point $(x_1, \ldots, x_n)$ in $F$ such that the complex
numbers $(x_1, \ldots, x_n)$ are algebraically equivalent over $k$.
\end{lem}

\begin{proof}
Since $k$ is countable, there exists only a countable number of algebraic subvarieties of $\C^n$ defined over $k$. By
our assumption on $F$, there exists a point of $F$ which does not lie in any proper algebraic variety defined over $k$.
Such a point has coordinates which are algebraically independent over $k$.
\end{proof}

Using the preceding lemma and induction, we can find a sequence of points
$$p_i=(x^i_1, \ldots, x^i_n)\in F$$
such that the $(x^i_j)_{i\in\N, j\leq n}$ are algebraically independent over $\Q$. Now let $(y^i_j)_{i\in\N, j\leq n}$
be a sequence of points in $\C^n$ such that $\{(y^i_1, \ldots, y^i_n), i\in\N\}$ is dense in $\C^n$. We can find an
automorphism $\sigma$ of $\C$ mapping $x^i_j$ to $y^i_j$ for all $i,j$. The closed subset of $\C^n$ $\sigma(F)$ contains
a dense subset of $\C^n$, hence $\sigma(F)=\C^n$. This shows that $F=\C^n$ and concludes the proof.
\end{proof}

As an immediate consequence, we get the following.

\begin{prop}
Let $X$ be a complex projective variety, and let $F$ be a closed subset of $X^{an}$. Assume that for any automorphism
$\sigma$
of $\C$, the subset
$$\sigma(F)\subset X^{\sigma}(\C)$$
is closed in $(X^{\sigma})^{an}$. Then $F$ is an algebraic subvariety of $X$.
\end{prop}

Given that absolute Hodge classes are classes in the singular cohomology groups that are, in some sense, Galois
equivariant\footnote{See \cite{Del82}, Question 2.4, where the questions of whether those are the only ones is raised.}
and that by the preceding result, algebraic subvarieties are the only closed subsets with a good behavior with respect
to the Galois action, this might serve as a motivation for Conjecture \ref{AH-implies-alg}.

\bigskip

Another, more precise, reason that explains why Conjecture \ref{AH-implies-alg} might be more tractable than the Hodge
conjecture is given by the work of Andr\'e around motivated cycles in \cite{An}. Through motivic considerations, Andr\'e
does indeed show that for most of the absolute Hodge classes we know, Conjecture \ref{AH-implies-alg} is actually a
consequence of the standard conjectures, which, at least in characteristic zero, seem considerably weaker than the Hodge
conjecture.

While we won't prove such results, it is to be noted that the proofs we gave in Paragraphs \ref{variations} and
\ref{standard} were given so as to imply Andr\'e's results for the absolute Hodge classes we will consider. The
interested reader should have no problem filling the gaps.

\bigskip

In the following sections, we will not use the notation $X^{an}$ for the complex manifold underlying a complex variety
$X$ anymore, but rather, by an abuse of notation, use $X$ to refer to both objects. The context will hopefully help the
reader avoid any confusion.

\section{Absolute Hodge classes in families}

This section deals with the behavior of absolute Hodge classes under deformation. We will focus on consequences of the
algebraicity of Hodge bundles. We prove Deligne's principle B,
which states that absolute Hodge classes are preserved by parallel transport, and discuss the link between Hodge loci
and absolute Hodge classes as in \cite{Voabs}. The survey \cite{VoHodge} contains a beautiful account of similar
results.

We only work here with projective families. The quasi-projective case could be dealt with and provide similar results,
see for instance \cite{ChAJ}.

\subsection{The variational Hodge conjecture and the global invariant cycle theorem}

Before stating Deligne's Principle B in \cite{Del82}, let us explain a variant of the Hodge conjecture.

Let $S$ be a smooth connected quasi-projective variety, and let $\pi : \mathcal X\ra S$ be a smooth projective morphism.
Let $0$ be a complex point of $S$, and, for some integer $p$ let $\alpha$ be a cohomology class in $H^{2p}(\mathcal X_0,
\Q(p))$. Assume that $\alpha$ is the cohomology class of some codimension $p$ algebraic cycle $Z_0$, and that $\alpha$
extends as a section $\widetilde{\alpha}$ of the local system $R^{2p}\pi_*\Q(p)$ on $S$\footnote{Once again, the
use of a Tate twist might seem superfluous. We write it down in order to emphasize the variation of Hodge structures
that this local system underlies, as well as the factor $(2 \pi i)^{-p}$ that appears in the comparison isomorphism with
de Rham cohomology.}.

In \cite{Gr66}, footnote 13, Grothendieck makes the following conjecture.

\begin{conj}(Variational Hodge conjecture) \label{varH}
For any complex point $s$ of $S$, the class $\widetilde{\alpha}_s$ is the cohomology class of an algebraic cycle.
\end{conj}

Using the Gauss-Manin connection and the isomorphism between de Rham and singular cohomology, we can formulate an
alternative version of the variational Hodge conjecture in de Rham cohomology. For this, keeping the notations as above,
we have a coherent sheaf $\mathcal H^{2p}=\R^{2p}\pi_* \Omega^{\bullet}_{\mathcal X/S}$ which computes the relative de
Rham cohomology of $\mathcal X$ over $S$. As we saw earlier, it is endowed with a canonical connection, the Gauss-Manin
connection $\nabla$.

\begin{conj}\label{varH-dR}(Variational Hodge conjecture for de Rham cohomology)
Let $\beta$ be a cohomology class in $H^{2p}(\mathcal X_0/\C)$. Assume that $\beta$ is the cohomology class of some
codimension $p$ algebraic cycle $Z_0$, and that $\beta$ extends as a section $\widetilde{\beta}$ of the coherent sheaf
$\mathcal H^{2p}=\R^{2p}\pi_* \Omega^{\bullet}_{\mathcal X/S}$ such that $\widetilde{\beta}$ is flat for the Gauss-Manin
connection. The variational Hodge conjecture states that for any complex point $s$ of $S$, the class
$\widetilde{\beta}_s$ is the cohomology class of an algebraic cycle.
\end{conj}

\begin{prop}\label{varH-equiv}
 Conjecture \ref{varH} and \ref{varH-dR} are equivalent.
\end{prop}

\begin{proof}
The de Rham comparison isomorphism between singular and de Rham cohomology in a relative context takes the form
of a canonical isomorphism
\begin{equation}\label{relative-dR}
\R^{2p}\pi_* \Omega^{\bullet}_{\mathcal X/S}\simeq R^{2p}\pi_*\Q(p)\otimes_{\Q} \mathcal O_S.
\end{equation}

Note that this formula is not one from algebraic geometry. Indeed, the sheaf $\mathcal O_S$ denotes here the sheaf of
holomorphic functions on the complex manifold $S$. The derived functor $\R^{2p}\pi_*$ on the left is a functor between
categories of complexes of holomorphic coherent sheaves, while the one on the right is computed for sheaves with the
usual complex topology. The Gauss-Manin connection is the connection on $\R^{2p}\pi_* \Omega^{\bullet}_{\mathcal X/S}$
for which the local system $R^{2p}\pi_*\Q(p)$ is constant. As we saw earlier, the locally free sheaf
$\R^{2p}\pi_* \Omega^{\bullet}_{\mathcal X/S}$ is algebraic, as well as the Gauss-Manin connection.

Given $\beta$ a cohomology class in the de Rham cohomology group $H^{2p}(\mathcal X_0/\C)$ as above,
we know that $\beta$ belongs to the rational subspace $H^{2p}(\mathcal X_0, \Q(p))$ because it is the cohomology class
of an algebraic cycle. Furthermore, since $\widetilde{\beta}$ is flat for the Gauss-Manin connection and is rational at
one point, it corresponds to a section of the local system $R^{2p}\pi_*\Q(p)$ under the comparison isomorphism above.
This shows that Conjecture \ref{varH} implies Conjecture \ref{varH-dR}.

On the other hand, sections of the local system $R^{2p}\pi_*\Q(p)$ induce flat algebraic sections of the coherent sheaf
$\R^{2p}\pi_* \Omega^{\bullet}_{\mathcal X/S}$. By definition of the Gauss-Manin connection, they do induce flat
holomorphic sections, and we have to show that those are algebraic. This is a consequence of the following important
result, which is due to Deligne.

\begin{thm}(Global invariant cycle theorem)\label{global}
Let $\pi : \mathcal X\ra S$ be a smooth projective morphism of quasi-projective complex varieties, and let $i : \mathcal
X\hookrightarrow\overline{\mathcal X}$ be a smooth compactification of $\mathcal X$. Let $0$ be complex point of $S$,
and let $\pi_1(S, 0)$ be the fundamental group of $S$. For any integer $k$, the space of monodromy-invariant classes of
degree $k$
$$H^k(\mathcal X_0, \Q)^{\pi_1(S, 0)}$$
is equal to the image of the restriction map
$$i_0^* : H^k(\overline{\mathcal X}, \Q)\ra H^k(\mathcal X_0, \Q),$$
where $i_0$ is the inclusion of $\mathcal X_0$ in $\overline{\mathcal X}$.
\end{thm}
In the theorem, the monodromy action is the action of the fundamental group $\pi_1(S, 0)$ on the cohomology groups of
the fiber $\mathcal X_0$. Note that the theorem implies that the space $H^k(\mathcal X_0, \Q)^{\pi_1(S, 0)}$ is a
sub-Hodge structure of $H^k(\mathcal X_0, \Q)$. However, the fundamental group of $S$ does not in general act by
automorphisms of Hodge structures.

The global invariant cycle theorem implies the algebraicity of flat holomorphic sections of the vecotr
bundle $\R^{2p}\pi_*
\Omega^{\bullet}_{\mathcal X/S}$ as follows. Let $\widetilde{\beta}$ be such a section, and keep the notation
of the theorem. By definition of the Gauss-Manin connection, $\widetilde{\beta}$ corresponds to a section of the local
system $R^{2p}\pi_*\C(p)$ under the isomorphism \ref{relative-dR}, that is, to a monodromy-invariant class in
$H^{2p}(\mathcal X_0, \C(p))$. The global invariant cycle theorem shows, using the comparison theorem between singular
and de Rham cohomology on $\overline{\mathcal X}$, that $\widetilde{\beta}$ comes from a de Rham cohomology class $b$
in $H^{2p}(\overline{\mathcal X}/\C)$. As such, it is algebraic.

The preceding remarks readily show the equivalence of the two versions of the variational Hodge conjecture.
\end{proof}

\bigskip

The next proposition shows that the variational Hodge conjecture is actually a part of the Hodge conjecture. This fact
is a consequence of the global invariant cycle theorem. The following proof will be rewritten in the next paragraph to
give results on absolute Hodge cycles.

\begin{prop}\label{varH-prop}
Let $S$ be a smooth connected quasi-projective variety, and let $\pi : \mathcal X\ra S$ be a smooth projective morphism.
Let $0$ be a complex point of $S$, and let $p$ be an integer.
\begin{enumerate}
\item Let $\alpha$ be a cohomology class in $H^{2p}(\mathcal X_0,
\Q(p))$. Assume that $\alpha$ is a Hodge class and that $\alpha$
extends as a section $\widetilde{\alpha}$ of the local system $R^{2p}\pi_*\Q(p)$ on $S$. Then for any complex point $s$
of $S$, the classes
$\widetilde{\alpha}_s$ is a Hodge class.
\item Let $\beta$ be a cohomology class in $H^{2p}(\mathcal X_0/\C)$. Assume that $\beta$ is a Hodge class and that
$\beta$ extends as a section $\widetilde{\beta}$ of the coherent sheaf
$R^{2p}\pi_* \Omega^{\bullet}_{\mathcal X/S}$ such that $\widetilde{\beta}$ is flat for the Gauss-Manin connection. Then
for any complex point $s$ of $S$, the classes
$\widetilde{\beta}_s$ is a Hodge class.
\end{enumerate}
\end{prop}

As an immediate corollary, we get the following.

\begin{cor}
The Hodge conjecture implies the variational Hodge conjecture.
\end{cor}

\begin{proof}
The two statements of the proposition are equivalent by the same arguments as in Proposition \ref{varH-equiv}.
Let us keep the notations as above. We want to prove that for any complex point $s$ of $S$, the class
$\widetilde{\alpha}_s$ is a Hodge class. Let us show how this is a consequence of the global invariant cycle theorem.
This is a simple consequence of Corollary \ref{antecedent} in the -- easier -- context of Hodge classes. Let us prove
the result from scratch.

As in Proposition \ref{polarization}, we can find a pairing
$$H^{2p}(\overline{\mathcal X}, \Q)\otimes H^{2p}(\overline{\mathcal X}, \Q)\ra \Q(1)$$
which turns $H^{2p}(\overline{\mathcal X}, \Q)$ into a polarized Hodge structure.

Let $i : \mathcal X\hookrightarrow\overline{\mathcal X}$ be a smooth compactification of $\mathcal X$, and let $i_0$ be
the inclusion of $\mathcal X_0$ in $\overline{\mathcal X}$.

By the global invariant cycle theorem, the morphism
$$i_0^* : H^{2p}(\overline{\mathcal X}, \Q)\ra H^{2p}(\mathcal X_0, \Q)^{\pi_1(S, 0)}$$
is surjective. It restricts to an isomorphism of Hodge structures
$$i_0^* : (\mathrm{Ker} i_0^*)^{\bot}\ra H^{2p}(\mathcal X_0, \Q)^{\pi_1(S, 0)},$$
hence a Hodge class $a\in(\mathrm{Ker} i_0^*)^{\bot}\subset H^{2p}(\overline{\mathcal X}, \Q)$ mapping to $\alpha$.
Indeed, saying that $\alpha$ extends to a global section of the local system $R^{2p}\pi_*\Q(p)$ exactly means that
$\alpha$ is monodromy-invariant, and $\alpha$ is a Hodge class since it is a cohomology class of an algebraic cycle.

Now let $i_s$ be the inclusion of $\mathcal X_s$ in $\overline{\mathcal X}$. Since $S$ is connected, we have
$\widetilde{\alpha}_s=i_s^*(a),$
which shows that $\widetilde{\alpha}_s$ is a Hodge class, and that the Hodge conjecture implies it is the cohomology
class of an algebraic cycle.
\end{proof}

It is an important fact that the variational Hodge conjecture is a purely algebraic statement. Indeed, we saw earlier
that both relative de Rham cohomology and the Gauss-Manin connection can be defined algebraically. This is to be
compared to the above discussion of the transcendental aspect of the Hodge conjecture, where one cannot avoid to use
singular cohomology, which cannot be defined in a purely algebraic fashion as it does depend on the topology of $\C$.

Very little seems to be known about the variational Hodge conjecture, see however \cite{Blo}.

\subsection{Deligne's Principle B}

In this paragraph, we state and prove the so-called Principle B for absolute Hodge cycles, which is due to Deligne. It
shows that the variational Hodge conjecture is true if one replaces algebraic cohomology classes by absolute Hodge
classes.

\begin{thm}\label{principle-B}(Principle B)
Let $S$ be a smooth connected quasi-projective variety, and let $\pi : \mathcal X\ra S$ be a smooth projective morphism.
Let $0$ be a complex point of $S$, and, for some integer $p$ let $\alpha$ be a cohomology class in $H^{2p}(\mathcal X_0,
\Q(p))$. Assume that $\alpha$ is an absolute Hodge class and that $\alpha$ extends as a section $\widetilde{\alpha}$ of
the local system $R^{2p}\pi_*\Q(p)$ on $S$. Then for any complex point $s$ of $S$, the class $\widetilde{\alpha}_s$ is
absolute Hodge.
\end{thm}

As in Proposition \ref{varH-equiv}, this is equivalent to the following rephrasing.

\begin{thm}\label{principle-BdR}(Principle B for de Rham cohomology)
Let $S$ be a smooth connected quasi-projective variety, and let $\pi : \mathcal X\ra S$
be a smooth projective morphism. Let $0$ be a complex point of $S$, and, for some integer $p$ let $\beta$ be a
cohomology class in $H^{2p}(\mathcal X_0/\C)$. Assume that $\beta$ is an absolute Hodge class and that $\beta$ extends
as a flat section $\widetilde{\beta}$ of the locally free sheaf
$\mathcal H^{2p} = \R^{2p}\pi_* \Omega^{\bullet}_{\mathcal X/S}$ endowed with the Gauss-Manin connection. Then for any
complex point $s$ of $S$, the class $\widetilde{\beta}_s$ is absolute Hodge.
\end{thm}

We will give two different proofs of this result to illustrate the techniques we introduced earlier. Both rely on
Proposition \ref{varH-prop}, and on the global invariant cycle theorem. The first one proves the result as a consequence
of the algebraicity of the Hodge bundles and of the Gauss-Manin connection. It is essentially Deligne's proof
in \cite{Del82}. The second proof elaborates on polarized Hodge structures and is inspired by Andr\'e's approach in
\cite{An}.

\begin{proof}
We will work with de Rham cohomology. Let $\sigma$ be an automorphism of $\C$. Since $\widetilde{\beta}$ is a global
section of the locally free sheaf $\mathcal H^{2p}$, we can form the conjugate section
$\widetilde{\beta}^{\sigma}$ of the conjugate sheaf $(\mathcal H^{2p})^{\sigma}$ on
$S^{\sigma}$. Now as in \ref{definitions}, this sheaf identifies with the relative de Rham cohomology of $\mathcal
X^{\sigma}$ over $S^{\sigma}$.

Fix a complex point $s$ in $S$, the class $\widetilde{\beta}_s$ is absolute Hodge. This means that for any such
$\sigma$, the class $\widetilde{\beta}^{\sigma}_{\sigma(s)}$ is a Hodge class. Now since $\beta=\widetilde{\beta}_0$ is
an absolute Hodge class by assumption, $\widetilde{\beta}^{\sigma}_{\sigma(0)}$ is a Hodge class.

Since the construction of the Gauss-Manin connection commutes with base change, the Gauss-Manin connection
$\nabla^{\sigma}$ on the relative de Rham cohomology of $\mathcal X^{\sigma}$ over $S^{\sigma}$ is the conjugate by
$\sigma$ of the Gauss-Manin connection on $\mathcal H^{2p}$.

These remarks allow us to write
$$\nabla^{\sigma}\widetilde{\beta}^{\sigma}=(\nabla\widetilde{\beta})^{\sigma}=0$$
since $\widetilde{\beta}$ is flat. This shows that $\widetilde{\beta}^{\sigma}$ is a flat section of the relative de
Rham cohomology of $\mathcal X^{\sigma}$ over $S^{\sigma}$. Since $\widetilde{\beta}^{\sigma}_{\sigma(0)}$ is a Hodge
class, Proposition \ref{varH-prop} shows that $\widetilde{\beta}^{\sigma}_{\sigma(0)}$ is a Hodge class, which is what
we needed to prove.
\end{proof}

Note that while the above proof may seem just a formal computation, it actually uses in an essential way the important
fact that both relative de Rham cohomology and the Gauss-Manin connection are algebraic object, which makes it possible
to conjugate them by field automorphisms.

Let us give a second proof of Principle B.

\begin{proof}
This is a consequence of Corollary \ref{antecedent}. Indeed, let $i : \mathcal X\hookrightarrow\overline{\mathcal X}$ be
a smooth compactification of $\mathcal X$, and let $i_0$ be the inclusion of $\mathcal X_0$ in $\overline{\mathcal X}$.

By the global invariant cycle theorem, the morphism
$$i_0^* : H^{2p}(\overline{\mathcal X}, \Q)\ra H^{2p}(\mathcal X_0, \Q)^{\pi_1(S, 0)}$$
is surjective. As a consequence, since $\alpha$ is monodromy-invariant, it belongs to the image of $i_0^*$. By Corollary
\ref{antecedent}, we can find an absolute Hodge class $a\in H^{2p}(\overline{\mathcal X}, \Q)$ mapping to $\alpha$.
Now let $i_s$ be the inclusion of $\mathcal X_s$ in $\overline{\mathcal X}$. Since $S$ is connected, we have
$$\widetilde{\alpha}_s=i_s^*(a),$$
which shows that $\widetilde{\alpha}_s$ is an absolute Hodge class, and concludes the proof.
\end{proof}

Note that following the remarks we made around the notion of motivated cycles, this argument could be used to prove that
the standard conjectures imply the variational Hodge conjecture, see \cite{An}.

\bigskip

Principle B will be one of our main tools in proving that some Hodge classes are absolute. When working with families
of varieties, it allows us to work with specific members of the family where algebraicity results might be known. When
proving the algebraicity of the Kuga-Satake correspondence between a projective K3 surface and its Kuga-Satake abelian
variety, it will make it possible to reduce to the case of Kummer surfaces, while in the proof of Deligne's theorem
that Hodge classes on abelian varieties are absolute, it allows for a reduction to the case of abelian varieties with
complex multiplication. Its mixed case version is instrumental to the results of \cite{ChAJ}.

\subsection{The locus of Hodge classes}

In this paragraph, we recall the definitions of the Hodge locus and the locus of Hodge classes associated to a
variation of Hodge structures and discuss their relation to the Hodge conjecture. The study of those has been started by
Griffiths in \cite{Gr68}. References on this subject include
\cite{Vo-book}, chapter 17 and \cite{VoHodge}. To simplify matters, we will only deal with variations of Hodge
structures coming from geometry, that is, coming from the cohomology of a family of smooth projective varieties. We will
point out statements that generalize to the quasi-projective case.

\bigskip

Let $S$ be a smooth complex quasi-projective variety, and let $\pi : \mathcal X\ra S$ be a smooth projective morphism.
Let $p$ be an integer. As earlier, consider the Hodge bundles
$$\mathcal H^{2p} = \R^{2p}\pi_*\Omega^{\bullet}_{\mathcal X/S}$$
together with the Hodge filtration
$$F^k\mathcal H^{2p} = \R^{2p}\pi_*\Omega^{\bullet\geq k}_{\mathcal X/S}.$$

Those are algebraic vector bundles over $S$, as we saw before. They are endowed with the Gauss-Manin connection
$$\nabla : \mathcal H^{2p} \ra \mathcal H^{2p}\otimes \Omega^1_{\mathcal X/S}.$$
Furthermore, the local system
$$H^{2p}_{\Q}=R^{2p}\pi_*\Q(p)$$
injects into $\mathcal H^{2p}$ and is flat with respect to the Gauss-Manin connection.

Let us start with a set-theoretic definition of the locus of Hodge classes.
\begin{df}
 The locus of Hodge classes in $\mathcal H^{2p}$ is the set of pairs $(\alpha,s)$, $s\in S(\C)$, $\alpha\in \mathcal
H^{2p}_s$, such that $\alpha$ is a Hodge class, that is, $\alpha\in  F^p\mathcal H^{2p}_s$ and $\alpha\in
H^{2p}_{\Q,s}$.
\end{df}

It turns out that the locus of Hodge classes is the set of complex points of a countable union of analytic subvarieties
of $\mathcal H^{2p}$. This can be seen as follows, see the above references for a thorough description. Let $(\alpha,
s)$ be in the locus of Hodge classes of $\mathcal H^{2p}$. We want to describe the component of the locus of Hodge
classes passing through $(\alpha, s)$ as an analytic variety in a neighborhood of $(\alpha, s)$.

On a neighborhood of $s$, the class $\widetilde{\alpha}$ extends to a flat holomorphic section of $\mathcal H^{2p}$. Now
the points $(\widetilde{\alpha}_t, t)$, for $t$ in the neighborhood of $s$, which belong to the locus of Hodge classes
are the points of an analytic variety, namely the variety defined by the $(\widetilde{\alpha}_t, t)$ such that
$\widetilde{\alpha}_t$ vanishes in the holomorphic (and even algebraic) vector bundle $\mathcal H^{2p}/F^p\mathcal
H^{2p}$.

It follows from this remark that the locus of Hodge classes is a countable union of analytic subvarieties of $\mathcal
H^{2p}$. Note that if we were to consider only integer cohomology classes to define the locus of Hodge classes, we would
actually get an analytic subvariety. The locus of Hodge classes was introduced in \cite{CDK}. It is of course very much
related to the more classical Hodge locus.

\begin{df}
The Hodge locus associated to $\mathcal H^{2p}$ is the projection on $S$ of the locus of Hodge classes. It is a
countable union of analytic subvarieties of $S$.
\end{df}

Note that the Hodge locus is interesting only when $\mathcal H^{2p}$ has no flat global section of type $(p,p)$. Indeed,
if it has, the Hodge locus is $S$ itself. However, in this case, one can always split off any constant variation of
Hodge structures for $\mathcal H^{2p}$ and consider the Hodge locus for the remaining variation of Hodge structures.

The reason why we are interested in those loci is the way they are related to the Hodge conjecture. Indeed, one has the
following.

\begin{prop}\label{alg}
If the Hodge conjecture is true, then the locus of Hodge classes and the Hodge locus for $\mathcal H^{2p}\ra S$ are
countable unions of closed algebraic subsets of $\mathcal H^{2p}\times S$ and $S$ respectively.
\end{prop}

\begin{proof}
We only have to prove the proposition for the locus of Hodge classes. If the Hodge conjecture is true, the locus of
Hodge classes is the locus of cohomology classes of algebraic cycles with rational coefficients. Those algebraic cycles
are parametrized by Hilbert schemes for the family $\mathcal X/B$. Since those are proper and have countably many
connected components, the Hodge locus is a countable union of closed algebraic subsets of $\mathcal H^{2p}\times S$.
\end{proof}

This consequence of the Hodge conjecture is now a theorem proved in \cite{CDK}.

\begin{thm}(Cattani, Deligne, Kaplan)
With the notations above, the locus of Hodge classes and the Hodge locus for $\mathcal H^{2p}\ra S$ are countable unions
of closed algebraic subsets of $\mathcal H^{2p}\times S$ and $S$ respectively.
\end{thm}

As before, the preceding discussion can be led in the quasi-projective case. Generalized versions of the Hodge
conjecture lead to similar algebraicity predictions, and indeed the corresponding algebraicity result for variations of
mixed Hodge structures is proved in \cite{BPS}, after the work of Brosnan-Pearlstein on the zero
locus of normal functions in \cite{BP}.

\subsection{Galois action on relative de Rham cohomology}

Let $S$ be a smooth irreducible quasi-projective variety over a field $k$, and let $\pi : \mathcal X\ra S$ be a smooth
projective morphism. Let $p$ be an integer. Consider again the Hodge bundles $\mathcal H^{2p}$
together with the Hodge filtration $F^k\mathcal H^{2p} = R^{2p}\pi_*\Omega^{\bullet\geq k}_{\mathcal X/S}.$ Those are
defined over $k$.

\bigskip

Let $\alpha$ be a section of $\mathcal H^{2p}$ over $S$. Let $\eta$ be the generic point of $S$. The class $\alpha$
induces a class $\alpha_{\eta}$ in the de Rham cohomology of the generic fiber $\mathcal X_{\eta}$ of $\pi$.

Let $\sigma$ be any embedding of $k(S)$ in $\C$ over $k$. The
morphism $\sigma$ corresponds to a morphism Spec$(\C)\ra\eta\ra S$, hence it induces a complex point $s$ of $S_{\C}$.
We have an isomorphism
$$\mathcal X_{\eta}\times_{k(S)}\C \simeq \mathcal X_{\C,s}$$
and the cohomology class $\alpha_{\eta}$ pulls back to a class $\alpha_{s}$ in the cohomology of $\mathcal
X_{\C,s}$.

The class $\alpha_s$ only depends on the complex point $s$. Indeed, it can be obtained the following way. The class
$\alpha$ pulls-back as a section $\alpha_{\C}$ of $\mathcal H^{2p}_{\C}$ over $S_{\C}$. The class $\alpha_s$ is the
value of $\alpha_{\C}$ at the point $s\in S(\C)$.

The following rephrases the definition of a Hodge class.
\begin{prop}\label{Hodgevar}
Assume that $\alpha_{\eta}$ is an absolute Hodge class. If $\alpha_{\eta}$ is absolute, then $\alpha_{s}$ is a Hodge
class. Furthermore, in case $k=\Q$, $\alpha_{\eta}$ is absolute if and only if $\alpha_s$ is a Hodge class for all $s$
induced by embeddings $\sigma : \Q(S)\ra \C$.
\end{prop}

\bigskip
We try to investigate the implications of the previous rephrasing.
\begin{lem}\label{Baire}
Assume the field $k$ is countable. Then the set of points $s\in S_{\C}(\C)$ induced by embeddings of $k(S)$ in $\C$ over
$k$ is dense in $S_{\C}(\C)$ for the usual topology.
\end{lem}

\begin{proof}
 Say that a complex point of $S_{\C}$ is very general if it does not lie in any proper algebraic subset of $S_{\C}$
defined over $k$. Since $k$ is countable, the Baire theorem shows that the set of general points is dense in
$S_{\C}(\C)$ for the usual topology.

 Now consider a very general point $s$. There exists an embedding of $k(S)$ into $\C$ such that the associated complex
point of $S_{\C}$ is $s$. Indeed, $s$ being very general exactly means that the image of the morphism
 $$\xymatrix{\mathrm{Spec}(\C)\ar[r]^s & S_{\C} \ar[r] & S}$$
 is $\eta$, the generic point of $S$, hence a morphism Spec$(\C)\ra\eta$ giving rise to $s$. This concludes the proof of
the lemma.
\end{proof}
We say that a complex point of $S_{\C}$ is very general if it lies in the aforementioned subset.

\begin{thm}\label{flat}
Let $S$ be a smooth irreducible quasi-projective variety over a subfield $k$ of $\C$ with generic point $\eta$, and let
$\pi : \mathcal X\ra S$ be a smooth projective
morphism. Let $p$ be an integer, and let $\alpha$ be a section of $\mathcal H^{2p}$ over $S$.
\begin{enumerate}
 \item Assume the class $\alpha_{\eta}\in H^{2p}(\mathcal X_{\eta}/k(S))$ is absolute Hodge. Then $\alpha$ is flat for
the Gauss-Manin connection and $\alpha_{\C}$ is a Hodge class at every complex point of $S_{\C}$.
 \item Assume that $k=\Q$. Then the class $\alpha_{\eta}\in H^{2p}(\mathcal X_{\eta}/\Q(S))$ is absolute Hodge if and
only if $\alpha$ is flat for the Gauss-Manin connection and for any connected component $S'$ of $S_{\C}$, there exists
a complex point $s$ of $S'$ such that $\alpha_s$ is a Hodge class.
\end{enumerate}
\end{thm}

\begin{proof}
All the objects we are considering are defined over a subfield of $k$ that is finitely generated over $\Q$, so we can
assume that $k$ is finitely generated over $\Q$, hence countable. Let $\alpha_{\C}$ be the section of $\mathcal
H^{2p}_{\C}$ over $S_{\C}$ obtained by pulling-back $\alpha$. The value of the class $\alpha_{\C}$ at any general point
is a Hodge class. Locally on $S_{\C}$, the bundle $\mathcal H^{2p}_{\C}$ with the Gauss-Manin connection is
biholomorphic to the flat bundle $S\times\C^n$, $n$ being the rank of $\mathcal H^{2p}_{\C}$, and we can assume such a
trivialization respects the rational subspaces.

Under such trivializations, the section $\alpha_{\C}$ is given locally on $S_{\C}$ by $n$ holomorphic functions which
take rational values on a dense subset. It follows that $\alpha_{\C}$ is locally constant, that is, that $\alpha_{\C}$,
hence $\alpha$, is flat for the Gauss-Manin connection. Since $\alpha$ is absolute Hodge, $\alpha_{\C}$ is a Hodge
class at any very general point of $S_{\C}$. Since those are dense in $S_{\C}(\C)$, Proposition \ref{varH-prop} shows
that $\alpha_{\C}$ is a Hodge class at every complex point of $S_{\C}$. This proves the first part of the theorem.

For the second part, assuming $\alpha$ is flat for the Gauss-Manin connection and $\alpha_s$ is Hodge for points $s$ in
,all the connected components of $S_{\C}$, Proposition \ref{varH-prop} shows that $\alpha_s$ is a Hodge class at all
the complex points $s$ of $S_{\C}$. In particular, this true for the general points of $S_{\C}$, which proves that
$\alpha_{\eta}$ is an absolute Hodge class by Proposition \ref{Hodgevar}.
\end{proof}

As a corollary, we get the following important result.
\begin{thm}\label{rigid}
Let $k$ be an algebraically closed subfield of $\C$, and let $X$ be a smooth projective variety over $k$. Let $\alpha$
be an absolute Hodge class of degree $2p$ in $X_{\C}$. Then $\alpha$ is defined over $k$, that is, $\alpha$ is the
pull-back of an absolute Hodge class in $X$.
\end{thm}

\begin{proof}
The cohomology class $\alpha$ belongs to $H^{2p}(X_{\C}/\C)=H^{2p}(X/k)\otimes \C$.  We need to show that it lies in
$H^{2p}(X/k)\subset H^{2p}(X_{\C}/\C)$, that is, that it is defined over $k$.

The class $\alpha$ is defined over a field $K$ finitely generated over $k$. Since $K$ is generated by a finite number of
elements over $k$, we can find a smooth irreducible quasi-projective variety $S$ defined over $k$ such that $K$ is
isomorphic to $k(S)$. Let $\mathcal X=X\times S$, and let $\pi$ be the projection of $\mathcal X$ onto $S$. Saying that
$\alpha$ is defined over $k(S)$ means that $\alpha$ is a class defined at the generic fiber of $\pi$. Up to replacing
$S$ by a Zariski-open subset, we can assume that $\alpha$ extends to a section $\widetilde{\alpha}$ of the relative de
Rham cohomology group $\mathcal H^{2p}$ of $\mathcal X$ over $S$. Since $\alpha$ is an absolute Hodge class, Theorem
\ref{flat} shows that $\widetilde{\alpha}$ is flat with respect to the Gauss-Manin connection on $\mathcal H^{2p}$.

Since $\mathcal X=X\times S$, relative de Rham cohomology is trivial, that is, the flat bundle $\mathcal H^{2p}$ is
isomorphic to $H^{2p}(X/k)\otimes \mathcal O_S$ with the canonical connection. Since $\widetilde{\alpha}$ is a flat
section over $S$ which is irreducible over the algebraically closed field $k$, it corresponds to the constant section
with value some $\alpha_0$ in $H^{2p}(X/k)$. Then $\alpha$ is the image of $\alpha_0$ in
$H^{2p}(X_{\C}/\C)=H^{2p}(X/k)\otimes \C$, which concludes the proof.
\end{proof}

\begin{rk}
In case $\alpha$ is the cohomology class of an algebraic cycle, the preceding result is a consequence of the existence
of Hilbert schemes. If $Z$ is an algebraic cycle in $X_{\C}$, $Z$ is algebraically equivalent to an algebraic cycle
defined over $k$. Indeed, $Z$ corresponds to a point in some product of Hilbert schemes parameterizing subschemes of
$X$.
Those Hilbert schemes are defined over $k$, so their points with value in $k$ are dense. This shows the result. Of
course, classes of algebraic cycles are absolute Hodge, so this is a special case of the previous result.
\end{rk}

\subsection{The field of definition of the locus of Hodge classes}

In this paragraph, we present some of the results of Voisin in \cite{Voabs}. While those could be proved using
Principle B and the global invariant cycle theorem along a line of arguments we used earlier, we focus on deducing the
theorems as consequences of statements from the previous paragraph. The reader can consult \cite{VoHodge} for the
former approach.

\bigskip

Let $S$ be a smooth complex quasi-projective variety, and let $\pi : \mathcal X\ra S$ be a smooth
projective morphism. Let $p$ be an integer, and let $\mathcal H^{2p}=\R^{2p}\pi_*\Omega^{\bullet}_{\mathcal X/S}$
together with the Hodge filtration $F^k\mathcal H^{2p} = \R^{2p}\pi_*\Omega^{\bullet\geq k}_{\mathcal X/S}.$ Assume
$\pi$ is defined over $\Q$. Then $\mathcal H^{2p}$ is defined over $\Q$, as well as the Hodge
filtration. Inside $\mathcal H^{2p}$, we have the locus of Hodge classes as before. It is an algebraic subset of
$\mathcal H^{2p}$.

Note that any smooth projective complex variety is isomorphic to the fiber of such a morphism $\pi$ over a complex
point. Indeed, if $X$ is a smooth projective complex variety, it is defined over a field finitely generated over $\Q$.
Noticing that such a field is the function field of a smooth quasi-projective variety $S$ defined over $\Q$ allows us
to find $\mathcal X\ra S$ as before. Note that $S$ might not be geometrically irreducible.

\begin{thm}
 Let $s$ be a complex point of $S$, and let $\alpha$ be a Hodge class in $H^{2p}(\mathcal X_s/\C)$. Then $\alpha$ is an
absolute Hodge class if and only if the connected component $Z_{\alpha}$ of the locus of Hodge classes passing through
$\alpha$ is defined over $\overline{\Q}$ and the conjugates of $Z_{\alpha}$ by $\mathrm{Gal}(\overline{\Q}/\Q)$ are
contained in the locus of Hodge classes.
\end{thm}

\begin{proof}
 Let $Z'$ be the smallest algebraic subset defined over $\Q$ containing $Z_{\alpha}$. It is the $\Q$-Zariski closure of
$Z_{\alpha}$. We want to show that $Z'$ is contained in the locus of Hodge classes if and only if $\alpha$ is absolute
Hodge.

Pulling back to the image of $Z'$ and spreading the base scheme $S$ if necessary, we can reduce to the
situation where $Z'$ dominates $S$, and there exists a smooth projective morphism
$$\pi_{\Q} : \mathcal X_{\Q} \ra S_{\Q}$$
defined over $\Q$, such that $\pi$ is the pull-back of $\pi_{\Q}$ to $\C$, a class $\alpha_{\Q}$ in
$H^{2p}(\mathcal X_{\Q}/S))$, and a
embedding of $\Q(S_{\Q})$ into $\C$ corresponding to the complex point $s\in S(\C)$ such that $\mathcal X_s$ and
$\alpha$ are the pullback of $\mathcal X_{\Q, \eta}$ and $\alpha_{\eta}$ respectively, where $\eta$ is the generic point
of $S$.

In this situation, by the definition of absolute Hodge classes, $\alpha$ is an absolute Hodge class if and only if
$\alpha_{\eta}$ is. Also, since $Z'$ dominates $S$, $Z'$ is contained in the locus of Hodge classes if and only if
$\alpha$ extends as a flat section of $\mathcal H^{2p}$ over $S$ which is a Hodge class at every complex point. Such a
section is automatically defined over $\Q$ since the Gauss-Manin connection is. Statement (2) of Theorem \ref{flat}
allows us to conclude the proof.
\end{proof}

\begin{rk}
 It is to be noted that the proof uses in an essential way the theorem of Cattani-Deligne-Kaplan on the algebraicity of
Hodge loci.
\end{rk}

Recall that Conjecture \ref{H-implies-AH} predicts that Hodge classes are absolute. As an immediate consequence, we get
the following reformulation.

\begin{cor}
 Conjecture \ref{H-implies-AH} is equivalent to the following.

Let $S$ be a smooth complex quasi-projective variety, and let $\pi : \mathcal X\ra S$ be a smooth projective morphism.
Assume $\pi$ is defined over $\Q$. Then the locus of Hodge classes for $\pi$ is a countable union of algebraic
subsets of the Hodge bundles defined over $\Q$.
\end{cor}

It is possible to prove the preceding corollary without resorting to the Cattani-Deligne-Kaplan theorem, see
Proposition \ref{alg}.

In the light of this result, the study of whether Hodge classes are absolute can be seen as a study of the field of
definition of the locus of Hodge classes. An intermediate property is to a ask for the component of the locus of Hodge
classes passing through a class $\alpha$ to be defined over $\overline{Q}$. In \cite{Voabs}, Voisin shows how one can
use arguments from the theory of variations of Hodge structures to give infinitesimal criteria for this to happen.

This is closely related to the rigidity result of Theorem \ref{rigid}. Indeed, using the fact that the Gauss-Manin
connection is defined over $\Q$, it is easy to show that the component of the locus of Hodge classes passing through a
class $\alpha$ in the cohomology of a complex variety defined over $\overline{\Q}$ is defined over $\overline{Q}$ if and
only if $\alpha$ is defined over $\overline{\Q}$ as a class in algebraic de Rham cohomology.

\bigskip

Let us conclude this section by showing how the study of fields of definition for Hodge loci is related to the Hodge
conjecture. The following is due to Voisin in \cite{Voabs}.

\begin{thm}
 Let $S$ be a smooth complex quasi-projective variety, and let $\pi : \mathcal X\ra S$ be a smooth projective morphism.
Assume $\pi$ is defined over $\Q$. Let $s$ be a complex point of $S$ and let $\alpha$ be a Hodge class in
$H^{2p}(\mathcal X_s, \Q(p)$. If the image in $S$ of the component of the locus of Hodge classes passing through
$\alpha$ is defined over $\overline{\Q}$, then the Hodge conjecture for $\alpha$ can be reduced to the Hodge conjecture
for varieties defined over number fields.
\end{thm}

\begin{proof}
 This is a consequence of the global invariant cycle theorem. Indeed, with the notation of Theorem \ref{global}, one
can choose the compactification $\overline{\mathcal X}$ to be defined over $\overline{\Q}$. The desired result follows
easily.
\end{proof}

\section{The Kuga-Satake construction}

In this section, we give our first nontrivial example of absolute Hodge classes. It is due to Deligne in \cite{DelK3}.

Let $S$ be a complex projective K3 surface. We want construct an abelian variety $A$ and an embedding of Hodge
structures
$$H^2(S, \Q)\hookrightarrow H^1(A, \Q)\otimes H^1(A)$$
which is absolute Hodge. This is the Kuga-Satake correspondence, see \cite{KS}, \cite{DelK3}.

We will take a representation-theoretic approach to this problem. This paragraph merely outlines the construction of the
Kuga-Satake correspondence, leaving aside part of the proofs. We refer to the survey \cite{vG} for more
details. Properties of Spin groups and their representations can be found in \cite{FH}, Chapter 20
or \cite{Bour}, Paragraph 9.

\subsection{Recollection on Spin groups}

We follow Deligne's approach in \cite{DelK3}. Let us start with some linear algebra. Let $V$ be a
finite-dimensional vector space over a field $k$ of characteristic zero with a non-degenerate quadratic form $Q$. Recall
that the Clifford algebra $C(V)$ over
$V$ is the algebra defined as the quotient of the tensor algebra $\bigoplus_{i\leq 0} V^{\otimes i}$ by the relation
$v\otimes v= Q(v), v\in V$. Even though the natural grading of the tensor algebra does not descend to the Clifford
algebra, there is a well-defined sub-algebra $C^+(V)$ of $C(V)$ which is the image of $\bigoplus_{i\leq 0} V^{\otimes
2i}$ in $C(V)$. The algebra $C^+(V)$ is the even Clifford algebra over $V$.

The Clifford algebra is endowed with an anti-automorphism $x\mapsto x^*$ such that $(v_1.\ldots v_i)^*=v_i.\ldots v_1$
if $v_1, \ldots, v_i\in V$. The Clifford group of $V$ is the algebraic group defined by
$$CSpin(V)=\{x\in C^+(V)^*, x.V.x^{-1}\subset V\}.$$
It can be proved that $CSpin(V)$ a connected algebraic group. By definition, it acts on $V$. Let $x\in CSpin(V), v\in
V$. We have $Q(xvx^{-1})=xvx^{-1}xvx^{-1}=xQ(v)x^{-1}=Q(v)$,
which shows that $CSpin(V)$ acts on $V$ through the orthogonal group $O(V)$, hence a map from $CSpin(V)$ to $O(V)$.
Since $CSpin(V)$ is connected, this map factors through $\tau : CSpin(V)\ra SO(V)$. We have an exact
sequence
$$\xymatrix{1\ar[r] & \mathbb G_m \ar[r]^w & CSpin(V)\ar[r] & SO(V)\ar[r] & 1}.$$

The spinor norm is the morphism of algebraic groups
$$N : CSpin(V)\ra \mathbb G_m, \,x\mapsto xx^*.$$
It is well-defined. Let $t$ be the inverse of $N$. The composite map
$$t\circ w : \mathbb G_m \ra \mathbb G_m$$
is the map $x\mapsto x^{-2}$. The Spin group $Spin(V)$ is the algebraic group defined as the kernel of $N$. The Clifford
group is generated by homotheties and elements of the Spin group.

The Spin group is connected and simply connected. The exact sequence
$$ 1\ra \pm 1 \ra Spin(V)\ra SO(V)\ra 1$$
realizes the Spin group as the universal covering of $SO(V)$.

\subsection{Spin representations}

The Clifford group has two different representations on $C^+(V)$. The first one is the adjoint representation
$C^+(V)_{ad}$. The adjoint action of $CSpin(V)$ is defined as
$$x._{ad} v =xvx^{-1},$$
where $x\in CSpin(V)$, $v\in C^+(V)$. It factors through $SO(V)$ and is isomorphic to $\bigoplus_i \bigwedge^{2i} V$ as
a representation of $CSpin(V)$.

The group $CSpin(V)$ acts on $C^+(V)$ by multiplication on the left, hence a representation $C^+(V)_s$, with
$$x._{s} v=xv,$$
where $x\in CSpin(V)$, $v\in C^+(V)$. It is compatible with the structure of right $C^+(V)$-module on $C^+(V)$, and we
have
$$\mathrm{End}_{C^+(V)}(C^+(V)_s)=C^+(V)_{ad}.$$

Assume $k$ is algebraically closed. We can describe those representations explicitly. In case the dimension of $V$ is
odd, let $W$ be a simple $C^+(V)$-module. The Clifford group $CSpin(V)$ acts on $W$. This is the spin representation of
$CSpin(V)$. Then $C^+(V)_s$ is isomorphic to a sum of copy of $W$, and $C^+(V)_{ad}$ is isomorphic to End$_k(W)$ as
representations of $CSpin(V)$.

In case the dimension of $V$ is even, let $W_1$ and $W_2$ be nonisomorphic simple $C^+(V)$-modules. Those are the
half-spin representations of $CSpin(V)$. Their sum $W$ is called the spin representation. Then $C^+(V)_s$ is isomorphic
to a sum of copy of $W$, and $C^+(V)_{ad}$ is isomorphic to End$_k(W_1)\times \mathrm{End}_k(W_2)$ as
representations of $CSpin(V)$.

\subsection{Hodge structures and the Deligne torus}\label{S}

Recall the definition of Hodge structures \`a la Deligne, see \cite{DelH2}. Let $S$ be the Deligne torus, that is, the
real algebraic group of
invertible elements of $\C$. It can be defined as the Weil restriction of $\mathbb G_m$ from $\C$ to $\R$. We have an
exact sequence of real algebraic groups
$$\xymatrix{ 1\ar[r] & \mathbb G_m \ar[r]^w & S\ar[r]^t & \mathbb G_m \ar[r] & 1},$$
where $w$ is the inclusion of $\R^*$ into $\C^*$ and $t$ maps a complex number $z$ to $|z|^{-2}$.  The composite map
$$t\circ w : \mathbb G_m \ra \mathbb G_m$$
is the map $x\mapsto x^{-2}$.

Let $V_{\Z}$ be a free $\Z$-module of finite rank, and let $V=V_{\Q}$. The datum of a Hodge structure of weight $k$ on
$V$ (or $V_{\Z}$) is the same as the datum of a representation
$\rho : S\ra GL(V_{\R})$
such that $\rho w (x) = x^k Id_{V_{\R}}$ for all $x\in \R^*.$ Given a Hodge structure of weight $n$, $z\in \C^*$ acts on
$V_{\R}$ by $z.v=z^p\overline{z}^q v$ if $v\in V^{p,q}$.

\subsection{From weight two to weight one}

Now assume $V$ is polarized of weight zero with Hodge numbers $V^{-1,1}=V^{1, -1}=1$, $V^{0,0}=0$. We say that $V$ (or
$V_{\Z}$) is of K3 type. We get a quadratic form $Q$ on $V_{\R}$, and the representation of $S$ on $V_{\R}$ factors
through the special orthogonal group of $V$ as
$h : S\ra SO(V_{\R})$.
\begin{lem}
 There exists a unique lifting of $h$ to a morphism $\widetilde{h} : S \ra CSpin(V_{\R})$ such that the following
diagram commutes.
$$\xymatrix{1\ar[r] & \mathbb G_m \ar[r]^w \ar@{=}[d] & S\ar[r]^t\ar[d]^{\widetilde{h}} & \mathbb G_m\ar[r]\ar@{=}[d] &
1 \\
            1\ar[r] & \mathbb G_m \ar[r]^w & CSpin(V_{\R})\ar[r]^t & \mathbb G_m \ar[r] & 1.}$$
\end{lem}
 \begin{proof}
It is easy to prove that such a lifting is unique if it exists. The restriction of $Q$ to $P=V_{\R}\bigcap
(V^{-1,1}\oplus V^{1, -1})$ is positive definite. Furthermore, $P$ has a canonical orientation. Let $e_1, e_2$ be a
direct orthonormal basis of $P$. We have $e_1e_2=-e_2e_1$ and $e_1^2=e_2^2=1$. As a consequence, $(e_2e_1)^2=-1$. An
easy computation shows that the morphism $a+ib\mapsto a+be_2e_1$ defines a suitable lifting of $h$.
\end{proof}

Using the preceding lemma, consider such a lifting $\widetilde{h} : S \ra CSpin(V)$ of $h$. Any representation of
$CSpin(V_{\R})$ thus gives rise to a Hodge structure. Let us first consider the adjoint representation. We know that
$C^+(V)_{ad}$ is isomorphic to $\bigoplus_i \bigwedge^{2i} V$, where $CSpin(V)$ acts on $V$ through $SO(V)$. It follows
that $\widetilde{h}$ endows $C^+(V)_{ad}$ with a weight zero Hodge structure. Since $V^{-1,1}=1$, the type of the Hodge
structure $C^+(V)_{ad}$ is $\{(-1,1), (0,0), (1,-1)\}$.

Now assume the dimension of $V$ is odd, and consider the spin representation $W$. It is a weight one
representation. Indeed, the lemma above shows that $C^+(V)_{s}$ is of weight one, and it is isomorphic to a sum of
copies of $W$. Since $C^+(V)_{ad}$ is isomorphic to End$_k(W)$ as representations of $CSpin(V)$, the type of $W$ is
$\{(1,0), (0,1)\}$.

It follows that $\widetilde{h}$ endows $C^+(V)_{s}$ with an effective Hodge structure of weight
one. It is possible to show that this Hodge structure is polarizable, see \cite{DelK3}. The underlying vector space
has $C^+(V_{\Z})$ as a natural lattice. This construction thus defines an abelian variety. Similar computations show
that the same result holds if the dimension of $V$ is even.

\begin{df}
The abelian variety defined by the Hodge structure on $C^+(V)_{s}$ with its natural lattice $C^+(V_{\Z})$ is called the
Kuga-Satake variety associated to $V_{\Z}$. We denote it by $KS(V_{\Z})$.
\end{df}

\bigskip

\begin{thm}\label{KS}
 Let $V_{\Z}$ be a polarized Hodge structure of K3 type. There exists a natural injective morphism of Hodge structures
$$V_{\Q}(-1)\hookrightarrow H^1(KS(V_{\Z}), \Q)\otimes H^1(KS(V_{\Z}), \Q).$$
This morphism is called the Kuga-Satake correspondence.
\end{thm}

\begin{proof}
Let $V=V_{\Q}$. Fix an element $v_0\in V$ and consider the vector space $M=C^+(V)$. It is endowed with a left action of
$V$ by the formula 
$$v.x=vxv_0$$
for $v\in V$, $x\in C^+(V)$. This action induces an embeddings
$$V\hookrightarrow \mathrm{End}_{\Q}(C^+(V)_s)$$
which is equivariant with respect to the action of $CSpin(V)$.

Now we can consider $\mathrm{End}_{\Q}(C^+(V)_s)(-1)$ as a subspace of $C^+(V)\otimes C^+(V)=H^1(KS(V_{\Z}),
\Q)\otimes H^1(KS(V_{\Z}), \Q)$ via a polarization of $KS(V_{\Z})$, and $V$ as a subspace of $C^+(V)$. This
gives an injection
$$V(-1)\hookrightarrow H^1(KS(V_{\Z}), \Q)\otimes H^1(KS(V_{\Z}), Q)$$
as desired. The equivariance property stated above shows that this is a morphism of Hodge structures.
\end{proof}

\begin{rk}
Let $V$ be a Hodge structure of K3 type. In order to construct the Kuga-Satake correspondence associated to $V$, we can
relax a bit the assumption that $V$ is polarized. Indeed, it is enough to assume that $V$ is endowed with a quadratic
form that is positive definite on $(V^{-1, 1}\oplus V^{1, -1})\bigcap V_{\R}$ and such that $V^{1, -1}$ and $V^{-1,
1}$ are totally isotropic subspaces of $V$. 
\end{rk}

\subsection{The Kuga-Satake correspondence is absolute}

Let $X$ be a polarized complex $K3$ surface. Denote by $KS(X)$ the Kuga-Satake variety associated to $H^2(X, \Z(1))$
endowed with the intersection pairing. Even though this pairing only gives a polarization on the primitive part of
cohomology, the construction is possible by the preceding remark. Theorem \ref{KS} gives us a correspondence between the
cohomology groups of $X$ and its Kuga-Satake variety. This is the Kuga-Satake correspondence for $X$. We can now state
and prove the main theorem of this section. It is proved by Deligne in \cite{DelK3}.

\begin{thm}
 Let $X$ be a polarized complex $K3$ surface. The Kuga-Satake correspondence
$$H^2(X, \Q(1))\hookrightarrow H^1(KS(X), \Q)\otimes H^1(KS(X), \Q)$$
is absolute Hodge.
\end{thm}

\begin{proof}
 Any polarized complex $K3$ surface deforms to a polarized Kummer surface in a polarized family. Now the Kuga-Satake
construction works in families. As a consequence, by Principle B, see Theorem \ref{principle-B}, it is enough to prove
that the Kuga-Satake correspondence is absolute Hodge for a variety $X$ which is the Kummer variety associated to an
abelian surface $A$. In this case, we can even prove the Kuga-Satake correspondence is algebraic. Let us outline the
proof of this result, which has first been proved by Morrison in \cite{Mor}. We follow a slightly different path.

\bigskip

First, remark that the canonical correspondence between $A$ and $X$ identifies the transcendental part of the Hodge
structure $H^2(X, \Z(1))$ with the transcendental part of $H^2(A, \Z(1))$. Note that the latter Hodge
structure is of K3 type. Since this isomorphism is induced by an algebraic correspondence between $X$ and $A$, standard
reductions show that it is enough to show that the Kuga-Satake correspondence between $A$ and the Kuga-Satake abelian
variety associated to $H^2(A, \Z(1))$ is algebraic. Let us write $U=H^1(A, \Q)$ and $V=H^2(A, \Q)$, considered as vector
spaces. 

We have $V=\bigwedge^2 U$. The vector space $U$ is of dimension $2$, and the weight $1$ Hodge structure on $U$ induces a
canonical isomorphism $\bigwedge^2 V=\bigwedge^4 U \simeq \Q$. The intersection pairing $Q$ on $V$ satisfies
$$\forall x, y\in V, Q(x,y)=x\wedge y.$$
Let $g\in SL(U)$. The determinant $g$ being $1$, $g$ acts trivially on $\bigwedge^2 V=\bigwedge^4 U$. As a consequence,
$g\wedge g$ preserves the intersection form on $V$. This gives a morphism $SL(U)\ra SO(V)$. The kernel of this morphism
is $\pm Id_U$, and it is surjective by dimension counting. Since $SL(U)$ is a connected algebraic group, this gives a
canonical isomorphism $SL(U)\simeq Spin(V)$.

The group $SL(U)$ acts on $U$ by the standard action and on its dual $U^*$ by $g\mapsto {}^tg^{-1}$. Those
representations are irreducible, and they are not isomorphic since no nontrivial bilinear form on $U$ is preserved by
$SL(U)$. By standard representation theory, those are the two half-spin representations of $SL(U)\simeq Spin(V)$. As a
consequence, the Clifford algebra of $V$ is canonically isomorphic to End$(U)\times \mathrm{End}(U^*)$, and we have a
canonical identification
$$CSpin(V)=\{(\lambda g,\lambda {}^tg^{-1}), g\in SL(U), \lambda\in \mathbb G_m\}.$$
An element $(\lambda g,\lambda {}^tg^{-1})$ of the Clifford group acts on the half-spin representations $U$ and $U^*$ by
the first and second component respectively.

\bigskip

The preceding identifications allow us to conclude the proof. Let $h' : S\ra GL(U)$ be the morphism that defines the
weight one Hodge structure on $U$, and let $h : S\ra SO(V)$ endow $V$ with its Hodge structure of K3 type. Note that if
$s\in \C^*$, the determinant of $h'(s)$ is $|s|^4$  since $U$ is of dimension $4$ and weight $1$. Since $V=\bigwedge^2
U(1)$ as Hodge structures, we get that $h$ is the morphism
$$h : s\mapsto |s|^{-2}h'(s)\wedge h'(s).$$
It follows that the morphism
$$\widetilde{h} : S\ra CSpin(V), s\mapsto (h'(s), |s|^2 \,{}^t h'(s)^{-1})=(|s||s|^{-1}h'(s)s,
|s|\,{}^t(|s|^{-1}h'(s))^{-1})$$
is a lifting of $h$ to $CSpin(V)$.

Following the previous identifications shows that the Hodge structure induced by $\widetilde{h}$ on $U$ and $U^*$ are
the ones induced by the identifications
$U=H^1(A, \Q)$ and $U^*=H^1(\hat{A}, \Q)$, where $\hat{A}$ is the dual abelian variety. Since the representation
$C^+(V)_s$ is a sum of $4$ copies of $U\oplus U^*$, this gives an isogeny between $KS(A)$ and $(A\times \hat{A})^4$ and
shows that the Kuga-Satake correspondence is algebraic, using the identity correspondence between $A$ and itself and
the correspondence between $A$ and its dual induced by the polarization. This concludes the proof.
\end{proof}

\begin{rk}
Since the cohomology of a Kummer variety is a direct factor of the cohomology of an abelian variety, it is an immediate
consequence of Deligne's theorem on absolute cycles on abelian varieties that the Kuga-Satake correspondence for Kummer
surfaces is absolute Hodge. However, our proof is more direct and also gives the
algebraicity of the correspondence in the Kummer case. Few algebraicity results are known for the Kuga-Satake
correspondence, but see \cite{Par} for the case of K3 surfaces which are a double cover of $\mathbb P^2$ ramified over
$6$ lines. See also \cite{vG} and \cite{VoKS} for further discussion of this problem.
\end{rk}

\begin{rk}
 In Definition \ref{def-etale}, we extended the notion of absolute Hodge classes to the setting of \'etale cohomology.
While we did not use this notion, most results we stated, for instance Principle B, can be generalized in this setting
with little additional work. This makes it possible to show that the Kuga-Satake correspondence is absolute Hodge in
the sense of Definition \ref{def-etale}. In the paper \cite{DelK3}, Deligne uses this to deduce the Weil conjectures
for K3 surfaces from the Weil conjectures for abelian varieties.
\end{rk}

\section{Deligne's theorem on Hodge classes on abelian varieties}


Having introduced the notion of absolute Hodge classes, Deligne went on to prove the
following remarkable theorem, which has already been mentioned several times in these
notes.

\begin{thm}[Deligne]
On an abelian variety, all Hodge classes are absolute.
\end{thm}

The purpose of the remaining lectures is to explain the proof of Deligne's theorem.
We follow Milne's account of the proof \cite{Del82}, with some simplifications due to
Andr\'e \cite{An-CM} and Voisin.

\subsection{Overview}

In the lectures of Griffiths and Kerr, we have already seen that rational Hodge
structures whose endomorphism algebra contains a CM-field are very special.  Since
abelian varieties of CM-type also play a crucial role in the proof of Deligne's
theorem, we shall begin by recalling two basic definitions.

\begin{df}
A \define{CM field} is a number field $E$ that admits an involution $\iota \in
\Gal(E/\QQ)$, such that for any embedding $s \colon E \into \CC$, one has
$\sb = s \circ \iota$.
\end{df}

The fixed field of the involution is a totally real field $F$; concretely, this means
that $F = \QQ(\alpha)$, where $\alpha$ and all of its conjugates are real numbers.
The field $E$ is then of the form $F \lbrack x \rbrack / (x^2 - f)$, for some element
$f \in F$ that is mapped to a negative number under all embeddings of $F$ into $\RR$.

\begin{df}
An abelian variety $A$ is said to be \define{of CM-type} if a CM-field $E$ is
contained in $\End(A) \tensor \QQ$, and if $H^1(A, \QQ)$ is one-dimensional as an
$E$-vector space. In that case, we clearly have $2 \dim A = \dim_{\QQ} H^1(A, \QQ) =
\Qdeg{E}$.
\end{df}

We will carry out a more careful analysis of abelian varieties and Hodge structures
of CM-type below. To motivate what follows, let us however briefly look at a criterion
for a simple abelian variety $A$ to be of CM-type that involves the (special)
Mumford-Tate group $\MT(A) = \MT \bigl( H^1(A) \bigr)$. 

Recall that the Hodge structure on $H^1(A, \QQ)$ can be
described by a morphism of $\RR$-algebraic groups $h \colon U(1) \to \GL \bigl(
H^1(A, \RR) \bigr)$; the weight being fixed, $h(z)$ acts as multiplication by
$z^{p-q}$ on the space $H^{p,q}(A)$. Recalling Paragraph \ref{S}, the group $U(1)$ is the kernel of the weight $w :
S\ra \mathbb{G}_m$. Representations of Ker$(w)$ correspond to Hodge structures of fixed weight.

We can define $\MT(A)$ as the smallest $\QQ$-algebraic
subgroup of $\GL \bigl( H^1(A, \QQ) \bigr)$ whose set of real points contains the
image of $h$. Equivalently, it is the subgroup fixing every Hodge class in every
tensor product
\[
	T^{p,q}(A) = H^1(A)^{\tensor p} \tensor H_1(A)^{\tensor q}.
\]
We have the following criterion.
\begin{prop}
A simple abelian variety is of CM-type if and only if its Mumford-Tate group $\MT(A)$
is an abelian group.
\end{prop}
Here is a quick outline of the proof of the interesting
implication.

\begin{proof}
Let $H = H^1(A, \QQ)$. The abelian variety $A$ is simple, which implies that
$E = \End(A) \tensor \QQ$ is a division algebra. It is also the space of
Hodge classes in $\End_{\QQ}(H)$, and therefore consists exactly of those endomorphisms
that commute with $\MT(A)$. Because the Mumford-Tate group is abelian, its action
splits $H^1(A, \CC)$ into a direct sum of character spaces
\[
	H \tensor_{\QQ} \CC = \bigoplus_{\chi} H_{\chi},
\]
where $m \cdot h = \chi(m) h$ for $h \in H_{\chi}$ and $m \in \MT(A)$. Now
any endomorphism of $H_{\chi}$ obviously commutes with $\MT(A)$, and is therefore
contained in $E \tensor_{\QQ} \CC$. By counting dimensions, we find that
\[
	\dim_{\QQ} E \geq \sum_{\chi} \bigl( \dim_{\CC} H_{\chi} \bigr)^2
		\geq \sum_{\chi} \dim_{\CC} H_{\chi} = \dim_{\QQ} H.
\]
On the other hand, we have $\dim_{\QQ} E \leq \dim_{\QQ} H$; indeed, since $E$ is a
division algebra, the map $E \to H$, $e \mapsto e \cdot h$, is injective for every
nonzero $h \in H$. Therefore $\Qdeg{E} = \dim_{\QQ} H = 2 \dim A$; moreover, each
character space $H_{\chi}$ is one-dimensional, and this implies that $E$ is
commutative, hence a field. To construct the involution $\iota \colon E \to E$ that
makes $E$ into a CM-field, choose a polarization $\psi \colon H \times H \to \QQ$,
and define $\iota$ by the condition that, for every $h,h' \in H$,
\[
	\psi(e \cdot h, h') = \psi \bigl( h, \iota(e) \cdot h' \bigr).
\]
The fact that $-i \psi$ is positive definite on the subspace $H^{1,0}(A)$ can then be
used to show that $\iota$ is nontrivial, and that $\sb = s \circ \iota$ for any
embedding of $E$ into the complex numbers.
\end{proof}

After this preliminary discussion of abelian varieties of CM-type, we return to
Deligne's theorem on an arbitrary abelian variety $A$. The proof consists of the
following three steps.
\begin{enumerate}[itemsep=1pt,leftmargin=18pt]
\renewcommand{\labelenumi}{\arabic{enumi}.}

\item The first step is to reduce the problem to abelian varieties of CM-type. This
is done by constructing an algebraic family of abelian varieties that links a given
$A$ and a Hodge class in $H^{2p}(A, \QQ)$ to an abelian variety of CM-type and a
Hodge class on it, and then applying Principle~B.
\item The second step is to show that every Hodge class on an abelian variety of
CM-type can be expressed as a sum of pullbacks of so-called split Weil classes. The
latter are Hodge classes on certain special abelian varieties, constructed
by linear algebra from the CM-field $E$ and its embeddings into $\CC$. This part of
the proof is due to Andr\'e \cite{An-CM}.
\item The last step is to show that all split Weil classes are absolute. For a fixed
CM-type, all abelian varieties of split Weil type are naturally parametrized by a
certain hermitian symmetric domain; by Principle~B, this allows to reduce the problem to
showing that split Weil classes on abelian varieties of a very specific form for which the result is straightforward.
\end{enumerate}

The original proof by Deligne uses Baily-Borel theory to show that certain families
of abelian varieties are algebraic. Following a suggestion by Voisin, we have chosen
to replace this by the following two results: the existence of a quasi-projective
moduli space for polarized abelian varieties with level structure and the theorem of
Cattani-Deligne-Kaplan in \cite{CDK} concerning the algebraicity of Hodge loci.

\subsection{Hodge structures of CM-type}

When $A$ is an abelian variety of CM-type, $H^1(A, \Q)$ is an example of a
Hodge structure of CM-type. We now undertake a more careful study of this class of
Hodge structures. Let $V$ be a rational Hodge structure of weight $n$, with Hodge
decomposition
\[
	V \tensor_{\QQ} \CC = \bigoplus_{p+q=n} V^{p,q}.
\]
Once we fix the weight $n$, there is a one-to-one correspondence between such
decompositions and group homomorphisms $h \colon U(1) \to \GL(V \tensor_{\QQ} \RR)$.
Namely, $h(z)$ acts as multiplication by $z^{p-q} = z^{2p-n}$ on the subspace
$V^{p,q}$. We define the (special) Mumford-Tate group $\MT(V)$ as the smallest
$\QQ$-algebraic subgroup of $\GL(V)$ whose set of real points contains the image of
$h$.

\begin{df}
We say that $V$ is a \define{Hodge structure of CM-type} if the following two
equivalent conditions are satisfied:
\begin{enumerate}[label=(\alph{*}),ref=(\alph{*})]
\item The set of real points of $\MT(V)$ is a compact torus.
\item $\MT(V)$ is abelian and $V$ is polarizable.
\end{enumerate}
\end{df}

It is not hard to see that any Hodge structure of CM-type is a direct sum of
irreducible Hodge structures of CM-type. Indeed, since $V$ is polarizable, it admits
a finite decomposition $V = V_1 \oplus \dotsb \oplus V_r$, with each $V_i$
irreducible. As subgroups of $\GL(V) = \GL(V_1) \times \dotsm
\times \GL(V_r)$, we then have $\MT(V) \subseteq \MT(V_1) \times \dotsm \times
\MT(V_r)$, and since the projection to each factor is surjective, it follows that
$\MT(V_i)$ is abelian. But this means that each $V_i$ is again of CM-type.
It is therefore sufficient to concentrate on irreducible Hodge structures of CM-type.
For those, there is a nice structure theorem that we shall now explain.

Let $V$ be an irreducible Hodge structure of weight $n$ that is of CM-type, and as
above, denote by $\MT(V)$ its special Mumford-Tate group. Because $V$ is irreducible,
its algebra of endomorphisms
\[
	E = \End_{\QHS}(V)
\]
must be a division algebra. In fact, since the endomorphisms of $V$ as a Hodge
structure are exactly the Hodge classes in $\End_{\QQ}(V)$, we see that $E$ consists
of all rational endomorphisms of $V$ that commute with $\MT(V)$. If $T_E = \Eun$
denotes the algebraic torus in $\GL(V)$ determined by $E$, then we get
$\MT(V) \subseteq T_E$ because $\MT(V)$ is commutative by assumption.

Since $\MT(V)$ is commutative, it acts on $V \tensor_{\QQ} \CC$ by characters, and so
we get a decomposition
\[
	V \tensor_{\QQ} \CC = \bigoplus_{\chi} V_{\chi},
\]
where $m \in \MT(V)$ acts on $v \in V_{\chi}$ by the rule $m \cdot v = \chi(m) v$.
Any endomorphism of $V_{\chi}$ therefore commutes with $\MT(V)$, and so $E
\tensor_{\QQ} \CC$ contains the spaces $\End_{\CC}(V_{\chi})$. This leads to the
inequality
\[
	\dim_{\QQ} E \geq \sum_{\chi} \bigl( \dim_{\CC} V_{\chi} \bigr)^2
		\geq \sum_{\chi} \dim_{\CC} V_{\chi} = \dim_{\QQ} V.
\]
On the other hand, we have $\dim_{\QQ} V \leq \dim_{\QQ} E$ because every nonzero
element in $E$ is invertible. It follows that each $V_{\chi}$ is one-dimensional,
that $E$ is commutative, and therefore that $E$ is a field of degree $\Qdeg{E} =
\dim_{\QQ} V$. In particular, $V$ is one-dimensional as an $E$-vector space.

The decomposition into character spaces can be made more canonical in the following
way. Let $S = \Hom(E, \CC)$ denote the set of all complex embeddings of $E$; its
cardinality is $\Qdeg{E}$. Then
\[
	E \tensor_{\QQ} \CC \xrightarrow{\sim} \bigoplus_{s \in S} \CC, \quad
		e \tensor z \mapsto \sum_{s \in S} s(e) z,
\]
is an isomorphism of $E$-vector spaces; $E$ acts on each summand on the right
through the corresponding embedding $s$. This decomposition induces an isomorphism
\[
	V \tensor_{\QQ} \CC \xrightarrow{\sim} \bigoplus_{s \in S} V_s,
\]
where $V_s = V \tensor_{E,s} \CC$ is a one-dimensional complex vector space on which
$E$ acts via $s$. The induced homomorphism $U(1) \to \MT(V) \to \Eun \to \End_{\CC}(V_s)$
is a character of $U(1)$, hence of the form $z \mapsto z^k$ for some integer $k$.
Solving $k = p-q$ and $n = p+q$, we find that $k = 2p-n$, which means that $V_s$ is
of type $(p,n-p)$ in the Hodge decomposition of $V$. Now define a function $\varphi
\colon S \to \ZZ$ by setting $\varphi(s) = p$; then any choice of isomorphism
$V \simeq E$ puts a Hodge structure of weight $n$ on $E$, whose Hodge decomposition
is given by
\[
	E \tensor_{\QQ} \CC
		\simeq \bigoplus_{s \in S} \CC^{\varphi(s), n - \varphi(s)}.
\]
From the fact that $\overline{e \tensor z} = e \tensor \bar{z}$, we deduce that
\[
	\overline{\sum_{s \in S} z_s} = \sum_{s \in S} \overline{z_{\sb}}.
\]
Since complex conjugation has to interchange $\CC^{p,q}$ and $\CC^{q,p}$, this
implies that $\varphi(\sb) = n - \varphi(s)$, and hence that
$\varphi(s) + \varphi(\sb) = n$ for every $s \in S$.

\begin{df}
Let $E$ be a number field, and $S = \Hom(E, \CC)$ the set of its complex embeddings.
Any function $\varphi \colon S \to \ZZ$ with the property that
$\varphi(s) + \varphi(\sb) = n$ defines a \define{Hodge structure $E_{\varphi}$ of
weight $n$} on the $\QQ$-vector space $E$, whose Hodge decomposition is given by
\[
	E_{\varphi} \tensor_{\QQ} \CC
		\simeq \bigoplus_{s \in S} \CC^{\varphi(s),\varphi(\sb)}.
\]
By construction, the action of $E$ on itself respects this decomposition.
\end{df}

In summary, we have $V \simeq E_{\varphi}$, which is an isomorphism both of
$E$-modules and of Hodge structures of weight $n$.  Next, we would like to prove that
in all interesting cases, $E$ must be a CM-field.  Recall that a field $E$ is called
a \define{CM-field} if there exists a nontrivial involution $\iota \colon E \to E$,
such that complex conjugation induces $\iota$ under any embedding of $E$ into the
complex numbers. In other words, we must have $s(\iota e) = \sb(e)$ for any $s \in S$
and any $e \in E$. We usually write $\bar{e}$ in place of $\iota e$, and refer to it
as complex conjugation on $E$. The fixed field of $E$ is then a totally real subfield
$F$, and $E$ is a purely imaginary quadratic extension of $F$.

To prove that $E$ is either a CM-field or $\QQ$, we choose a polarization $\psi$ on
$E_{\varphi}$.  We then define the so-called \define{Rosati involution} $\iota \colon
E \to E$ by the condition that
\[
	\psi(e \cdot x, y) = \psi(x, \iota e \cdot y)
\]
for every $x, y, e \in E$. Denoting the image of $1 \in E$ by $\sum_{s \in S} 1_s$,
we have
\[
	\sum_{s \in S} \psi(1_s, 1_{\sb}) s(e \cdot x) \sb(y)
		= \sum_{s \in S} \psi(1_s, 1_{\sb}) s(x) \sb(\iota e \cdot y),
\]
which implies that $s(e) = \sb(\iota e)$. Now there are two cases: Either $\iota$ is
nontrivial, in which case $E$ is a CM-field and the Rosati involution is complex
conjugation. Or $\iota$ is trivial, which means that $\sb = s$ for every complex
embedding. In the second case, we see that $\varphi(s) = n/2$ for every $s$, and so
the Hodge structure must be $\QQ(-n/2)$, being irreducible and of type $(n/2,n/2)$.
This implies that $E = \QQ$.

From now on, we exclude the trivial case $V = \QQ(-n/2)$ and assume that $E$ is a
CM-field.

\begin{df}
A \define{CM-type} of $E$ is a mapping $\varphi \colon S \to \{0, 1\}$ with the
property that $\varphi(s) + \varphi(\sb) = 1$ for every $s \in S$.
\end{df}

When $\varphi$ is a CM-type, $E_{\varphi}$ is the rational Hodge structure of an abelian
variety with complex multiplication by $E$ (unique up to isogeny). In general, we
have the following structure theorem.

\begin{prop}
Any Hodge structure $V$ of CM-type and of even weight $2k$ with $V^{p,q} = 0$ for $p
< 0$ or $q < 0$ occurs as a direct factor of $H^{2k}(A)$ for some abelian variety
with complex multiplication.
\end{prop}

\begin{proof}
In our classification of irreducible Hodge structures of CM-type above, there were
two cases: $\QQ(-n/2)$, and Hodge structures of the form $E_{\varphi}$, where
$E$ is a CM-field and $\varphi \colon S \to \ZZ$ is a function satisfying $\varphi(s)
+ \varphi(\sb) = n$. Clearly $\varphi$ can be written as a linear combination (with
integer coefficients) of CM-types for $E$. Because of the relations
\[
	E_{\varphi + \psi} \simeq E_{\varphi} \tensor_E E_{\psi} \quad \text{and} \quad
		\quad E_{-\varphi} \simeq E_{\varphi}^{\vee},
\]
every irreducible Hodge structure of CM-type can thus be obtained from Hodge
structures corresponding to CM-types by tensor products, duals, and Tate twists.

As we have seen, every Hodge structure of CM-type is a direct sum of irreducible
Hodge structures of CM-type. The assertion follows from this by simple linear
algebra.
\end{proof}

To conclude our discussion of Hodge structures of CM-type, we will consider the case
when the CM-field $E$ is a Galois extension of $\QQ$. In that case, the Galois group
$G = \Gal(E / \QQ)$ acts on the set of complex embeddings of $E$ by the rule
\[
	(g \cdot s)(e) = s(g^{-1} e).
\]
This action is simply transitive. Recall that we have an isomorphism
\[
	E \tensor_{\QQ} E \xrightarrow{\sim} \bigoplus_{g \in G} E, \quad
		x \tensor e \mapsto g(e) x.
\]
For any $E$-vector space $V$, this isomorphism induces a decomposition
\[
	V \tensor_{\QQ} E \xrightarrow{\sim} \bigoplus_{g \in G} V, \quad
		v \tensor e \mapsto g(e) v.
\]
When $V$ is an irreducible Hodge structure of CM-type, a natural question is whether
this decomposition is compatible with the Hodge decomposition. The following lemma shows
that the answer to this question is yes.

\begin{lem} \label{lem:Galois}
Let $E$ be a CM-field that is a Galois extension of $\QQ$, with Galois group $G =
\Gal(E / \QQ)$. Then for any $\varphi \colon S \to \ZZ$ with $\varphi(s) +
\varphi(\sb) = n$, we have
\[
	E_{\varphi} \tensor_{\QQ} E \simeq \bigoplus_{g \in G} E_{g \varphi}.
\]
\end{lem}

\begin{proof}
We chase the Hodge decompositions through the various isomorphisms that are involved
in the statement. To begin with, we have
\[
	\bigl( E_{\varphi} \tensor_{\QQ} E \bigr) \tensor_{\QQ} \CC
		\simeq \bigl( E_{\varphi} \tensor_{\QQ} \CC \bigr) \tensor_{\QQ} E
		\simeq \bigoplus_{s \in S} \CC^{\varphi(s), n-\varphi(s)} \tensor_{\QQ} E
		\simeq \bigoplus_{s,t \in S} \CC^{\varphi(s), n-\varphi(s)},
\]
and the isomorphism takes $(v \tensor e) \tensor z$ to the element
\[
	\sum_{s,t \in S} t(e) \cdot z \cdot s(v).
\]
On the other hand,
\[
	\bigl( E_{\varphi} \tensor_{\QQ} E \bigr) \tensor_{\QQ} \CC
		\simeq \bigoplus_{g \in G} E \tensor_{\QQ} \CC
		\simeq \bigoplus_{g \in G} \bigoplus_{s \in S} \CC^{\varphi(s), n-\varphi(s)},
\]
and under this isomorphism, $(v \tensor e) \tensor z$ is sent to the element
\[
	\sum_{g \in G} \sum_{s \in S} s(ge) \cdot s(v) \cdot z.
\]
If we fix $g \in G$ and compare the two expressions, we see that $t = sg$, and hence
\[
	E \tensor_{\QQ} \CC \simeq \bigoplus_{t \in S} \CC^{\varphi(s), n-\varphi(s)}
		\simeq \bigoplus_{t \in S} \CC^{\varphi(t g^{-1}), n - \varphi(t g^{-1})}.
\]
But since $(g \varphi)(t) = \varphi(t g^{-1})$, this is exactly the Hodge
decomposition of $E_{g \varphi}$.
\end{proof}

\subsection{Reduction to abelian varieties of CM-type}

The proof of Deligne's theorem involves the construction of algebraic families of
abelian varieties, in order to apply Principle~B. For this, we shall use the
existence of a fine moduli space for polarized abelian varieties with level
structure. Recall that if $A$ is an abelian variety of dimension $g$, the subgroup $A
\lbrack N \rbrack$ of its $N$-torsion points is isomorphic to
$(\ZZ / N \ZZ)^{\oplus 2g}$. A \define{level $N$-structure} is a choice of symplectic
isomorphism $A \lbrack N \rbrack \simeq (\ZZ / N \ZZ)^{\oplus 2g}$. Also recall that
a \define{polarization of degree $d$} on an abelian variety $A$ is a finite morphism
$\theta \colon A \to \hat{A}$ of degree $d$.

\begin{thm}
Fix integers $g, d \geq 1$. Then for any $N \geq 3$, there is a smooth quasi-projective
variety $\MgdN$ that is a fine moduli space for $g$-dimensional abelian varieties
with polarization of degree $d$ and level $N$-structure.  In particular, we have a
universal family of abelian varieties over $\MgdN$.
\end{thm}

The relationship of this result with Hodge theory is the following. Fix an abelian
variety $A$ of dimension $g$, with level $N$-structure and polarization $\theta
\colon A \to \hat{A}$. The polarization corresponds to an antisymmetric bilinear form
$\psi \colon H^1(A, \ZZ) \times H^1(A, \ZZ) \to \ZZ$ that polarizes the Hodge
structure; we shall refer to $\psi$ as a \define{Riemann form}. Define $V = H^1(A,
\QQ)$, and let $D$ be the corresponding period domain; $D$ parametrizes all
possible Hodge structures of type $\{(1,0), (0,1)\}$ on $V$ that are polarized by the form $\psi$.
Then $D$ is isomorphic to the universal covering space of the quasi-projective
complex manifold $\MgdN$.

We now turn to the first step in the proof of Deligne's theorem, namely the reduction
of the general problem to abelian varieties of CM-type. This is accomplished by the
following theorem and Principle~B, see Theorem \ref{principle-B}.

\begin{thm} \label{thm:Step1}
Let $A$ be an abelian variety, and let $\alpha \in H^{2p}(A, \QQ(p))$ be a Hodge class
on $A$. Then there exists a family $\pi \colon \famA \to B$ of abelian varieties,
with $B$ irreducible and quasi-projective, such that the following three things are true:
\begin{enumerate}[label=(\alph{*}),ref=(\alph{*})]
\item $\famA_0 = A$ for some point $0 \in B$.
	\label{en:Step1-a}
\item There is a Hodge class $\tilde{\alpha} \in H^{2p}(\famA, \QQ(p))$ whose
restriction to $A$ equals $\alpha$.
	\label{en:Step1-b}
\item For a dense set of $t \in B$, the abelian variety $\famA_t = \pi^{-1}(t)$ is of
CM-type.
	\label{en:Step1-c}
\end{enumerate}
\end{thm}

Before giving the proof, we shall briefly recall the following useful interpretation
for period domains. Say $D$ parametrizes all Hodge structures of weight $n$ on a
fixed rational vector space $V$ that are polarized by a given bilinear form
$\psi$. The set of real points of the group $G = \Aut(V, \psi)$ then acts transitively on
$D$ by the rule $(g H)^{p,q} = g \cdot H^{p,q}$, and so $D \simeq G(\RR) / K$.

Now points of $D$ are in one-to-one correspondence with homomorphisms of real
algebraic groups $h \colon U(1) \to G_{\RR}$, and we denote the Hodge structure
corresponding to $h$ by $V_h$. Then $V_h^{p,q}$ is exactly the subspace of $V
\tensor_{\QQ} \CC$ on which $h(z)$ acts as multiplication by $z^{p-q}$, and from
this, it is easy to verify that $g V_h = V_{g h g^{-1}}$. In other words, the points
of $D$ can be thought of as conjugacy classes of a fixed $h$ under the action of
$G(\RR)$.

\begin{proof}[Proof of Theorem~\ref{thm:Step1}]
After choosing a polarization $\theta \colon A \to \hat{A}$, we may assume that the
Hodge structure on $V = H^1(A, \QQ)$ is polarized by a Riemann form $\psi$. Let $G =
\Aut(V, \psi)$, and recall that $M = \MT(A)$ is the smallest $\QQ$-algebraic subgroup
of $G$ whose set of real points $M(\RR)$ contains the image of the homomorphism $h
\colon U(1) \to G(\RR)$.  Let $D$ be the period domain whose points parametrize all
possible Hodge structures of type $\{(1,0), (0,1)\}$ on $V$ that are polarized by the form $\psi$.
With $V_h = H^1(A)$ as the base point, we then have $D \simeq G(\RR) / K$; the points
of $D$ are thus exactly the Hodge structures $V_{g h g^{-1}}$, for $g \in G(\RR)$ arbitrary.

The main idea of the proof is to consider the Mumford-Tate domain
\[
	D_h = M(\RR) / K \cap M(\RR) \into D.
\]
By definition, $D_h$ consists of all Hodge structures of the form $V_{g h g^{-1}}$,
for $g \in M(\RR)$. As explained in Griffiths' lectures, these are precisely the
Hodge structures whose Mumford-Tate group is contained in $M$.

To find Hodge structures of CM-type in $D_h$, we appeal to a result by Borel.  Since
the image of $h$ is abelian, it is contained in a maximal torus $T$ of the real Lie
group $M(\RR)$. One can show that, for a generic element $\xi$ in the Lie algebra
$\mlie_{\RR}$, this torus is the stabilizer of $\xi$ under the adjoint action by
$M(\RR)$. Now $\mlie$ is defined over $\QQ$, and so there exist arbitrarily small
elements $g \in M(\RR)$ for which $\Ad(g) \xi = g \xi g^{-1}$ is rational.  The
stabilizer $g T g^{-1}$ of such a rational point is then a maximal torus in $M$ that
is defined over $\QQ$. The Hodge structure $V_{g h g^{-1}}$ is a point of the
Mumford-Tate domain $D_h$, and by definition of the Mumford-Tate group, we have
$\MT(V_{g h g^{-1}}) \subseteq T$. In particular, $V_{g h g^{-1}}$ is of CM-type,
because its Mumford-Tate group is abelian. This reasoning shows that $D_h$ contains a
dense set of points of CM-type.

To obtain an algebraic family of abelian varieties with the desired properties, we
can now argue as follows. Let $\mathcal{M}$ be the moduli space of abelian varieties
of dimension $\dim A$, with polarization of the same type as $\theta$, and level
$3$-structure. Then $\mathcal{M}$ is a smooth quasi-projective variety, and since it
is a fine moduli space, it carries a universal family. Now there are finitely many
Hodge tensors $\tau_1, \dotsc, \tau_r$ for $H^1(A)$, such that $M = \MT(A)$ is
exactly the subgroup of $G$ fixing every $\tau_i$. Let $B \subseteq \mathcal{M}$ be
the irreducible component of the Hodge locus of $\tau_1, \dotsc, \tau_r$ that
passes through the point $A$. By the theorem of Cattani-Deligne-Kaplan, $B$ is
again a quasi-projective variety. Let $\pi \colon \famA \to B$ be the restriction of
the universal family to $B$. Then \ref{en:Step1-a} is clearly
satisfied for this family.

Now $D$ is the universal covering space of $\mathcal{M}$, with the point $V_h =
H^1(A)$ mapping to $A$. By construction, the preimage of $B$ in $D$ is exactly the
Mumford-Tate domain $D_h$. Indeed, consider a Hodge structure $V_{g h g^{-1}}$ in the
preimage of $B$. By construction, every $\tau_i$ is a Hodge tensor for this Hodge
structure, which shows that $\MT(V_{g h g^{-1}})$ is contained in $M$. As explained
above, this implies that $V_{g h g^{-1}}$ belongs to $D_h$. Since $D_h$ contains a
dense set of Hodge structures of CM-type, \ref{en:Step1-c} follows. Since $B$ is also
contained in the Hodge locus of $\alpha$, we obtain \ref{en:Step1-b} after passing to
a finite cover.
\end{proof}

\subsection{Background on hermitian forms}

The second step in the proof of Deligne's theorem involves the construction of
special Hodge classes on abelian varieties of CM-type, the so-called split Weil
classes. This requires some background on hermitian forms, which we now provide.
Throughout, $E$ is a CM-field, with totally real subfield $F$ and complex conjugation
$e \mapsto \bar{e}$, and $S = \Hom(E, \CC)$ denotes the set of complex embeddings of
$E$. An element $\zeta \in \Eun$ is called \define{totally imaginary} if $\bar{\zeta}
= -\zeta$; concretely, this means that $\sb(\zeta) = - s(\zeta)$ for every complex
embedding $s$. Likewise, an element $f \in \Fun$ is said to be \define{totally
positive} if $s(f) > 0$ for every $s \in S$.

\begin{df}
Let $V$ be an $E$-vector space.
A $\QQ$-bilinear form $\phi \colon V \times V \to E$ is said to be
\define{$E$-hermitian} if $\phi(e \cdot v,w) = e \cdot \phi(v,w)$ and $\phi(v,w) =
\overline{\phi(w,v)}$ for every $v,w \in V$ and every $e \in E$.
\end{df}

Now suppose that $V$ is an $E$-vector space of dimension $d = \dim_E V$, and that
$\phi$ is an $E$-hermitian form on $V$. We begin by describing the numerical
invariants of the pair $(V, \phi)$. For any embedding $s \colon E \into \CC$, we
obtain a hermitian form $\phi_s$ (in the usual sense) on the complex vector space
$V_s = V \tensor_{E,s} \CC$. We let $a_s$ and $b_s$ be the dimensions of the maximal
subspaces where $\phi_s$ is, respectively, positive and negative definite.

A second invariant of $\phi$ is its discriminant. To define it, note that $\phi$
induces an $E$-hermitian form on the one-dimensional $E$-vector space $\wed{E}{d} V$,
which up to a choice of basis vector, is of the form $(x,y) \mapsto f x \bar{y}$. The
element $f$ belongs to the totally real subfield $F$, and a different choice of basis
vector only changes $f$ by elements of the form $\NmEF(e) = e \cdot \bar{e}$.
Consequently, the class of $f$ in $\Fun / \NmEF(\Eun)$ is well-defined, and is called
the \define{discriminant} of $(V, \phi)$. We denote it by the symbol $\disc \phi$.

Now suppose that $\phi$ is nondegenerate. Let $v_1, \dots, v_d$ be an orthogonal
basis for $V$, and set $c_i = \phi(v_i, v_i)$. Then we have $c_i \in \Fun$, and
\[
	a_s = \# \menge{i}{s(c_i) > 0} \quad \text{and} \quad
		b_s = \# \menge{i}{s(c_i) < 0}
\]
satisfy $a_s + b_s = d$. Moreover, we have
\[
	f = \prod_{i=1}^d c_i \mod \NmEF(\Eun);
\]
this implies that $\sgn \bigl( s(f) \bigr) = (-1)^{b_s}$ for every $s \in S$. The
following theorem by Landherr \cite{Landherr} shows that the discriminant
and the integers $a_s$ and $b_s$ are a complete set of invariants for $E$-hermitian
forms.

\begin{thm}[Landherr] \label{thm:Landherr}
Let $a_s, b_s \geq 0$ be a collection of integers, indexed by the set $S$, and let $f
\in \Fun / \NmEF(\Eun)$ be an arbitrary element.
Suppose that they satisfy $a_s + b_s = d$ and $\sgn \bigl( s(f) \bigr) = (-1)^{b_s}$
for every $s \in S$. Then there exists a nondegenerate $E$-hermitian form $\phi$ on
an $E$-vector space $V$ of dimension $d$ with these invariants; moreover, $(V, \phi)$
is unique up to isomorphism.
\end{thm}

This classical result has the following useful consequence.

\begin{cor} \label{cor:split}
If $(V, \phi)$ is nondegenerate, then the following two conditions are equivalent:
\begin{enumerate}[label=(\alph{*}),ref=(\alph{*})]
\item $a_s = b_s = d/2$ for every $s \in S$, and $\disc \phi = (-1)^{d/2}$.
	\label{en:split-a}
\item There is a totally isotropic subspace of $V$ of dimension $d/2$.
	\label{en:split-b}
\end{enumerate}
\end{cor}

\begin{proof}
If $W \subseteq V$ is a totally isotropic subspace of dimension $d/2$, then $v
\mapsto \phi(\argbl, v)$ induces an antilinear isomorphism $V/W \xrightarrow{\sim}
W^{\vee}$. Thus we can extend a basis $v_1, \dotsc, v_{d/2}$ of $W$ to a basis $v_1,
\dotsc, v_d$ of $V$, with the property that
\begin{align*}
	\phi(v_i, v_{i+d/2}) = 1   \qquad & \text{for $1 \leq i \leq d/2$,}	\\
	\phi(v_i, v_j) = 0         \qquad & \text{for $\abs{i-j} \neq d/2$.}
\end{align*}
We can use this basis to check that \ref{en:split-a} is satisfied. For the converse,
consider the hermitian space $(E^{\oplus d}, \phi)$, where 
\[
	\phi(x, y) = \sum_{1 \leq i \leq d/2}
		\bigl( x_i \bar{y}_{i+d/2} + x_{i+d/2} \bar{y}_i \bigr)
\]
for every $x,y \in E^{\oplus d}$. By Landherr's theorem, this space is (up to
isomorphism) the unique hermitian space satisfying \ref{en:split-a}, and it is easy
to see that it satisfies \ref{en:split-b}, too.
\end{proof}

\begin{df}
An $E$-hermitian form $\phi$ that satisfies the two equivalent conditions in
Corollary~\ref{cor:split} is said to be \define{split}.
\end{df}

We shall see below that $E$-hermitian forms are related to polarizations on Hodge
structures of CM-type. We now describe one additional technical result that shall be
useful in that context. Suppose that $V$ is a Hodge structure of type $\{(1,0), (0,1)\}$
that is of CM-type and whose endomorphism ring contains $E$; let $h \colon U(1) \to
\Eun$ be the corresponding homomorphism.
Recall that a \define{Riemann form} for $V$ is a $\QQ$-bilinear antisymmetric form
$\psi \colon V \times V \to \QQ$, with the property that
\[
	(x, y) \mapsto \psi \bigl( x, h(i) \cdot \bar{y} \bigr)
\]
is hermitian and positive definite on $V \tensor_{\QQ} \CC$. We only consider Riemann
forms whose Rosati involution induces complex conjugation on $E$; that is, which
satisfy
\[
	\psi(e v, w) = \psi(v, \bar{e} w).
\]

\begin{lem} \label{lem:phi-psi}
Let $\zeta \in \Eun$ be a totally imaginary element ($\bar{\zeta} = -\zeta$), and let
$\psi$ be a Riemann form for $V$ as above. Then there exists a unique $E$-hermitian
form $\phi$ with the property that $\psi = \TrEQ(\zeta \phi)$.
\end{lem}

We begin with a simpler statement.

\begin{lem}
Let $V$ and $W$ be finite-dimensional vector spaces over $E$, and let $\psi \colon V
\times W \to \QQ$ be a $\QQ$-bilinear form such that $\psi(ev,w) = \psi(v,ew)$ for
every $e \in E$. Then there exists a unique $E$-bilinear form $\phi$ such that
$\psi(v,w) = \TrEQ \phi(v,w)$.
\end{lem}

\begin{proof}
The trace pairing $E \times E \to \QQ$, $(x,y) \mapsto \TrEQ(xy)$, is nondegenerate.
Consequently, composition with $\TrEQ$ induces an injective homomorphism
\[
	\Hom_E \bigl( V \tensor_E W, E \bigr)
		\to \Hom_{\QQ} \bigl( V \tensor_E W, \QQ \bigr),
\]
which has to be an isomorphism because both vector spaces have the same dimension
over $\QQ$.  By assumption, $\psi$ defines a $\QQ$-linear map $V \tensor_E W \to
\QQ$, and we let $\phi$ be the element of $\Hom_E \bigl( V \tensor_E W, E \bigr)$
corresponding to $\psi$ under the above isomorphism.
\end{proof}

\begin{proof}[Proof of Lemma~\ref{lem:phi-psi}]
We apply the preceding lemma with $W = V$, but with $E$ acting on $W$ through complex
conjugation. This gives a sesquilinear form $\phi_1$ such that $\psi(x,y) = \TrEQ
\phi_1(x,y)$. Now define $\phi = \zeta^{-1} \phi_1$, so that we have $\psi(x,y) =
\TrEQ \bigl( \zeta \phi(x,y) \bigr)$. The uniqueness of $\phi$ is obvious from the
preceding lemma.

It remains to show that we have $\phi(y,x) = \overline{\phi(x,y)}$. Because $\psi$ is
antisymmetric, $\psi(y,x) = -\psi(x,y)$, which implies that
\[
	\TrEQ \bigl( \zeta \phi(y,x) \bigr) = - \TrEQ \bigl( \zeta \phi(x,y) \bigr)
		= \TrEQ \bigl( \bar{\zeta} \phi(x,y) \bigr).
\]
On replacing $y$ by $ey$, for arbitrary $e \in E$, we obtain
\[
	\TrEQ \bigl( \zeta e \cdot \phi(y,x) \bigr)
		= \TrEQ \bigl( \overline{\zeta e} \cdot \phi(x,y) \bigr).
\]
On the other hand, we have
\[
	\TrEQ \bigl( \zeta e \cdot \phi(y,x) \bigr)
		= \TrEQ \bigl( \overline{\zeta e \cdot \phi(y,x)} \bigr)
		= \TrEQ \bigl( \overline{\zeta e} \cdot \overline{\phi(y,x)} \bigr).
\]
Since $\overline{\zeta e}$ can be an arbitrary element of $E$, the nondegeneracy of
the trace pairing implies that $\phi(x,y) = \overline{\phi(y,x)}$.
\end{proof}

\subsection{Construction of split Weil classes}

Let $E$ be a CM-field; as usual, we let $S = \Hom(E, \CC)$ be the set of complex
embeddings; it has $\Qdeg{E}$ elements.

Let $V$ be a rational Hodge structure of type $\{(1,0), (0,1)\}$ whose
endomorphism algebra contains $E$. We shall assume that $\dim_E V = d$ is an even
number. Let $V_s = V \tensor_{E, s} \CC$. Corresponding to the decomposition
\[
	E \tensor_{\QQ} \CC \xrightarrow{\sim} \bigoplus_{s \in S} \CC, \quad
		e \tensor z \mapsto \sum_{s \in S} s(e) z,
\]
we get a decomposition
\[
	V \tensor_{\QQ} \CC \simeq \bigoplus_{s \in S} V_s.
\]
The isomorphism is $E$-linear, where $e \in E$ acts on the complex vector space $V_s$
as multiplication by $s(e)$. Since $\dim_{\QQ} V = \Qdeg{E} \cdot \dim_E V$, each
$V_s$ has dimension $d$ over $\CC$. By assumption, $E$ respects the Hodge
decomposition on $V$, and so we get an induced decomposition
\[
	V_s = V_s^{1,0} \oplus V_s^{0,1}.
\]
Note that $\dim_{\CC} V_s^{1,0} + \dim_{\CC} V_s^{0,1} = d$.

\begin{lem} \label{lem:Weil}
The rational subspace $\wed{E}{d} V \subseteq \wed{\QQ}{d} V$ is purely of type
$(d/2,d/2)$ if and only if $\dim_{\CC} V_s^{1,0} = \dim_{\CC} V_s^{0,1} = d/2$ for every $s \in S$.
\end{lem}

\begin{proof}
We have
\[
	\Bigl( \wed{E}{d} V \Bigr) \tensor_{\QQ} \CC
		\simeq \wed{E \tensor_{\QQ} \CC}{d} V \tensor_{\QQ} \CC
		\simeq \bigoplus_{s \in S} \wed{\CC}{d} V_s
		\simeq \bigoplus_{s \in S} \Bigl( \wed{\CC}{p_s} V_s^{1,0} \Bigr)
			\tensor \Bigl( \wed{\CC}{q_s} V_s^{0,1} \Bigr),
\]
where $p_s = \dim_{\CC} V_s^{1,0}$ and $q_s = \dim_{\CC} V_s^{0,1}$. The assertion
follows because the Hodge type of each summand is evidently $(p_s, q_s)$.
\end{proof}

We will now describe a condition on $V$ that guarantees that the space $\wed{E}{d}
V$ consists entirely of Hodge cycles.

\begin{df} \label{def:split-Weil-type}
Let $V$ be a rational Hodge structure of type $\{(1,0), (0,1)\}$ with
$E \into \End_{\QHS}(V)$ and $\dim_E V = d$. We say that $V$ is \define{of split
Weil type relative to $E$} if there exists an $E$-hermitian form $\phi$ on $V$ with a
totally isotropic subspace of dimension $d/2$, and a totally imaginary element $\zeta
\in E$, such that $\TrEQ(\zeta \phi)$ defines a polarization on $V$.
\end{df}

According to Corollary~\ref{cor:split}, the condition on the $E$-hermitian form $\phi$ is
the same as saying that the pair $(V, \phi)$ is split.

\begin{prop} \label{prop:split-Hodge}
If $V$ is of split Weil type relative to $E$, and $\dim_E V = d$ is even, then the space
\[
	\wed{E}{d} V \subseteq \wed{\QQ}{d} V
\]
consists of Hodge classes of type $(d/2, d/2)$.
\end{prop}

\begin{proof}
Since $\psi = \TrEQ(\zeta \phi)$ defines a polarization, $\phi$ is nondegenerate; by
Corollary~\ref{cor:split}, it follows that $(V, \phi)$ is split. Thus for any complex
embedding $s \colon E \into \CC$, we have $a_s = b_s = d/2$. Let $\phi_s$ be the
induced hermitian form on $V_s = V \tensor_{E,s} \CC$. By Lemma~\ref{lem:Weil},
it suffices to show that $\dim_{\CC} V_s^{1,0} = \dim_{\CC} V_s^{0,1} = d/2$. By
construction, the isomorphism
\[
	\alpha \colon V \tensor_{\QQ} \CC \xrightarrow{\sim} \bigoplus_{s \in S} V_s
\]
respects the Hodge decompositions on both sides. For any $v \in V$, we have
\[
	\psi(v,v) = \TrEQ \bigl( \zeta \phi(v,v) \bigr)
		= \sum_{s \in S} s(\zeta) \cdot s \bigl( \phi(v,v) \bigr)
		= \sum_{s \in S} s(\zeta) \cdot \phi_s(v \tensor 1, v \tensor 1).
\]
Now if we choose a nonzero element $x \in V_s^{1,0}$, then under the above isomorphism,
\[
	-s(\zeta) i \cdot \phi_s(x, \bar{x})
		= \psi \bigl( \alpha^{-1}(x), h(i) \cdot \overline{\alpha^{-1}(x)} \bigr) > 0
\]
Likewise, we have $s(\zeta) i \cdot \phi_s(x, \bar{x}) > 0$ for $x \in V_s^{0,1}$ nonzero.
Consequently, $\dim_{\CC} V_s^{1,0}$ and $\dim_{\CC} V_s^{0,1}$ must both be less
than or equal to $d/2 = a_s = b_s$; since their dimensions add up to $d$, we get the
desired result.
\end{proof}

\subsection{Andr\'e's theorem and reduction to split Weil classes}

The second step in the proof of Deligne's theorem is to reduce the problem from
arbitrary Hodge classes on abelian varieties of CM-type to Hodge classes of split
Weil type. This is accomplished by the following pretty theorem due to Yves Andr\'e in
\cite{An-CM}.

\begin{thm}[Andr\'e]
Let $V$ be a rational Hodge structure of type $\{(1,0), (0,1)\}$, which is of CM-type.
Then there exists a CM-field $E$, rational Hodge structures $V_{\alpha}$ of split
Weil type (relative to $E$), and morphisms of Hodge structure $V_{\alpha} \to V$,
such that every Hodge cycle $\xi \in \wed{\QQ}{2k} V$ is a sum of images of
Hodge cycles $\xi_{\alpha} \in \wed{\QQ}{2k} V_{\alpha}$ of split Weil type.
\end{thm}

\begin{proof}
Let $V = V_1 \oplus \dotsb \oplus V_r$, with $V_i$ irreducible; then each $E_i
= \End_{\QHS}(V_i)$ is a CM-field. Define $E$ to be the Galois closure of the
compositum of the fields $E_1, \dotsc, E_r$, so that $E$ is a CM-field which is
Galois over $\QQ$ with Galois group $G = \Gal(E / \QQ)$. After replacing $V$ by
$V \tensor_{\QQ} E$ (of which $V$ is a direct factor), we may assume without loss of
generality that $E_i = E$ for all $i$.

As before, let $S = \Hom(E, \CC)$ be the set of complex embeddings of $E$; we then
have a decomposition
\[
	V \simeq \bigoplus_{i \in I} E_{\varphi_i}
\]
for some collection of CM-types $\varphi_i$. Applying Lemma~\ref{lem:Galois}, we get
\[
	V \tensor_{\QQ} E \simeq \bigoplus_{i \in I} \bigoplus_{g \in G} E_{g \varphi_i}.
\]
Since each $E_{g \varphi_i}$ is one-dimensional over $E$, we get
\[
	\Bigl( \wed{\QQ}{2k} V \Bigr) \tensor_{\QQ} E
		\simeq \wed{E}{2k} V \tensor_{\QQ} E
		\simeq \wed{E}{2k} \bigoplus_{(i,g) \in I \times G} E_{g \varphi_i}
		\simeq \bigoplus_{\substack{\alpha \subseteq I \times G \\ \abs{\alpha} = 2k}}
			\bigotimes_{(i,g) \in \alpha} E_{g \varphi_i}
\]
where the tensor product is over $E$. If we now define Hodge structures of CM-type
\[
	V_{\alpha} = \bigoplus_{(i,g) \in \alpha} E_{g \varphi_i}
\]
for any subset $\alpha \subseteq I \times G$ of size $2k$, then $V_{\alpha}$ has
dimension $2k$ over $E$. The above calculation shows that
\[
	\Bigl( \wed{\QQ}{2k} V \Bigr) \tensor_{\QQ} E
		\simeq \bigoplus_{\alpha} \wed{E}{2k} V_{\alpha},
\]
which is an isomorphism both as Hodge structures and as $E$-vector spaces. Moreover,
since $V_{\alpha}$ is a sub-Hodge structure of $V \tensor_{\QQ} E$, we clearly have
morphisms $V_{\alpha} \to V$, and any Hodge cycle $\xi \in \wed{\QQ}{2k} V$ is a sum
of Hodge cycles $\xi_{\alpha} \in \wed{E}{2k} V_{\alpha}$.

It remains to see that $V_{\alpha}$ is of split Weil type whenever $\xi_{\alpha}$ is
nonzero. Fix a subset $\alpha \subseteq I \times G$ of size $2k$, with the property that
$\xi_{\alpha} \neq 0$. Note that we have
\[
	\wed{E}{2k} V_{\alpha} \simeq \bigotimes_{(i,g) \in \alpha} E_{g \varphi_i}
		\simeq E_{\varphi},
\]
where $\varphi \colon S \to \ZZ$ is the function
\[
	\varphi = \sum_{(i,g) \in \alpha} g \varphi_i
\]
The Hodge decomposition of $E_{\varphi}$ is given by
\[
	E_{\varphi} \tensor_{\QQ} \CC \simeq \bigoplus_{s \in S} \CC^{\varphi(s), \varphi(\sb)}.
\]
The image of the Hodge cycle $\xi_{\alpha}$ in $E_{\varphi}$ must be purely of type
$(k,k)$ with respect to this decomposition. But
\[
	\xi_{\alpha} \tensor 1 \mapsto \sum_{s \in S} s(\xi_{\alpha}),
\]
and since each $s(\xi_{\alpha})$ is nonzero, we conclude that $\varphi(s) = k$ for every $s \in
S$. This means that the sum of the $2k$ CM-types $g \varphi_i$, indexed by $(i,g) \in \alpha$, is
constant on $S$. We conclude by the criterion in Proposition~\ref{prop:criterion-split} that
$V_{\alpha}$ is of split Weil type.
\end{proof}

The proof makes use of the following criterion for a Hodge structure to be of split
Weil type. Let $\varphi_1, \dotsc, \varphi_d$ be CM-types attached to $E$.  Let $V_i
= E_{\varphi_i}$ be the Hodge structure of CM-type corresponding to $\varphi_i$, and define
\[
	V = \bigoplus_{i=1}^d V_i.
\]
Then $V$ is a Hodge structure of CM-type with $\dim_E V = d$.

\begin{prop} \label{prop:criterion-split}
If $\sum \varphi_i$ is constant on $S$, then $V$ is of split Weil type.
\end{prop}

\begin{proof}
To begin with, it is necessarily the case that $\sum \varphi_i = d/2$; indeed,
\[
	\sum_{i=1}^d \varphi_i(s) + \sum_{i=1}^d \varphi(\sb)
		= \sum_{i=1}^d \bigl( \varphi_i(s) + \varphi_i(\sb) \bigr) = d,
\]
and the two sums are equal by assumption. By construction, we have
\[
	V \tensor_{\QQ} \CC
		\simeq \bigoplus_{i=1}^d \bigl( E_{\varphi_i} \tensor_{\QQ} \CC \bigr)
		\simeq \bigoplus_{i=1}^d \bigoplus_{s \in S} \CC^{\varphi_i(s), \varphi_i(\sb)}.
\]
This shows that
\[
	V_s = V \tensor_{E,s} \CC \simeq \bigoplus_{i=1}^d \CC^{\varphi_i(s), \varphi_i(\sb)}.
\]
Therefore $\dim_{\CC} V_s^{1,0} = \sum \varphi_i(s) = d/2$, and likewise $\dim_{\CC}
V_s^{0,1} = \sum \varphi_i(\sb) = d/2$.

Next, we construct the required $E$-hermitian form on $V$. For each $i$, choose a
Riemann form $\psi_i$ on $V_i$, whose Rosati involution acts as complex conjugation
on $E$. Since $V_i = E_{\varphi_i}$, there exist totally imaginary elements $\zeta_i
\in \Eun$, such that
\[
	\psi_i(x, y) = \TrEQ \bigl( \zeta_i x \bar{y} \bigr)
\]
for every $x,y \in E$. Set $\zeta = \zeta_d$, and define $\phi_i(x, y) =
\zeta_i \zeta^{-1} x \bar{y}$, which is an $E$-hermitian form on $V_i$ with the
property that $\psi_i = \TrEQ(\zeta \phi_i)$.

For any collection of totally positive elements $f_i \in F$,
\[
	\psi = \sum_{i=1}^d f_i \psi_i
\]
is a Riemann form for $V$. As $E$-vector spaces, we have $V = E^{\bigoplus d}$,
and so we can define a nondegenerate $E$-hermitian form on $V$ by the rule
\[
	\phi(v, w) = \sum_{i=1}^d f_i \phi_i(v_i, w_i).
\]
We then have $\psi = \TrEQ(\zeta \phi)$. By the same argument as before, $a_s
= b_s = d/2$, since $\dim_{\CC} V_s^{1,0} = \dim_{\CC} V_s^{0,1} = d/2$.
By construction, the form $\phi$ is diagonalized, and so its discriminant is easily found to be
\[
	\disc \phi = \zeta^{-d} \prod_{i=1}^d f_i \zeta_i \mod \NmEF(\Eun).
\]
On the other hand, we know from general principles that, for any $s \in S$,
\[
	\sgn \bigl( s(\disc \phi) \bigr) = (-1)^{b_s} = (-1)^{d/2}.
\]
This means that $\disc \phi = (-1)^{d/2} f$ for some totally positive element $f \in
\Fun$. Upon replacing $f_d$ by $f_d f^{-1}$, we get $\disc \phi = (-1)^{d/2}$, which proves
that $(V, \phi)$ is split.
\end{proof}

\subsection{Split Weil classes are absolute}

The third step in the proof of Deligne's theorem is to show that split Weil classes
are absolute. We begin by describing a special class of abelian varieties of split
Weil type where this can be proved directly.

Let $V_0$ be a rational Hodge structure of even rank $d$ and type $\{(1,0), (0,1)\}$. Let
$\psi_0$ be a Riemann form that polarizes $V_0$, and $W_0$ a maximal isotropic
subspace of dimension $d/2$. Also fix an element $\zeta \in \Eun$ with $\bar{\zeta} =
-\zeta$.

Now set $V = V_0 \tensor_{\QQ} E$, with Hodge structure induced by the isomorphism
\[
	V \tensor_{\QQ} \CC
		\simeq V_0 \tensor_{\QQ} \bigl( E \tensor_{\QQ} \CC \bigr)
		\simeq \bigoplus_{s \in S} V_0 \tensor_{\QQ} \CC.
\]
Define a $\QQ$-bilinear form $\psi \colon V \times V \to \QQ$ by the formula
\[
	\psi(v_0 \tensor e, v_0' \tensor e') 
		= \TrEQ \bigl( e \overline{e'} \bigr) \cdot \psi_0(v_0, v_0').
\]
This is a Riemann form on $V$, for which $W = W_0 \tensor_{\QQ} E$ is an isotropic
subspace of dimension $d/2$. By Lemma~\ref{lem:phi-psi}, there is a unique
$E$-hermitian form $\phi \colon V \times V \to E$ such that $\psi = \TrEQ(\zeta
\phi)$. By Corollary~\ref{cor:split}, $(V, \phi)$ is split, and $V$ is therefore of split
Weil type. Let $A_0$ be an abelian variety with $H^1(A_0, \Q)=V_0$. The integral lattice of $V_0$ induces an integral
lattice in $V = V_0 \tensor_{\QQ} E$. We denote by $A_0\tensor_{\Q} E$ the corresponding abelian variety. It is of
split Weil type since $V$ is.

The next result, albeit elementary, is the key to proving that split Weil classes are absolute.
\begin{prop} \label{prop:point}
Let $A_0$ be an abelian variety with $H^1(A_0) = V_0$, and define $A = A_0
\tensor_{\QQ} E$. Then the subspace $\wed{E}{d} H^1(A, \QQ)$ of $H^d(A, \QQ)$
consists entirely of absolute Hodge classes.
\end{prop}

\begin{proof}
We have $H^d(A, \QQ) \simeq \wed{\QQ}{d} H^1(A, \QQ)$, and the subspace
\[
	\wed{E}{d} H^1(A, \QQ) \simeq \wed{E}{d} V_0 \tensor_{\QQ} E
		\simeq \left( \wed{\QQ}{d} V_0 \right) \tensor_{\QQ} E
		\simeq H^d(A_0, \QQ) \tensor_{\QQ} E
\]
consists entirely of Hodge classes by Proposition~\ref{prop:split-Hodge}. But since
$\dim A_0 = d/2$, the space $H^d(A_0, \QQ)$ is generated by the fundamental class of
a point, which is clearly absolute. This implies that every class in $\wed{E}{d}
H^1(A, \QQ)$ is absolute.
\end{proof}

The following theorem, together with Principle~B as in Theorem \ref{principle-B}, completes the proof of Deligne's
theorem. 

\begin{thm} \label{thm:Step3}
Let $E$ be a CM-field, and let $A$ be an abelian variety of split Weil type (relative
to $E$). Then there exists a family $\pi \colon \famA \to B$ of abelian varieties,
with $B$ irreducible and quasi-projective, such that the following three things are true:
\begin{enumerate}[label=(\alph{*}),ref=(\alph{*})]
\item $\famA_0 = A$ for some point $0 \in B$.
\item For every $t \in B$, the abelian variety $\famA_t = \pi^{-1}(t)$ is of split
Weil type (relative to $E$).
\item The family contains an abelian variety of the form $A_0 \tensor_{\QQ} E$.
\end{enumerate}
\end{thm}

The proof of Theorem~\ref{thm:Step3} takes up the remainder of this section.
Throughout, we let $V = H^1(A, \QQ)$, which is an $E$-vector space of some even
dimension $d$. The polarization on $A$ corresponds to a Riemann form $\psi \colon V
\times V \to \QQ$, with the property that the Rosati involution acts as complex
conjugation on $E$. Fix a totally imaginary element $\zeta \in \Eun$; then $\psi =
\TrEQ(\zeta \phi)$ for a unique $E$-hermitian form $\phi$ by Lemma~\ref{lem:phi-psi}.
Since $A$ is of split Weil type, the pair $(V, \phi)$ is split.

As before, let $D$ be the period domain, whose points parametrize Hodge structures of
type $\{(1,0), (0,1)\}$ on $V$ that are polarized by the form $\psi$. Let $\Dsp \subseteq
D$ be the subset of those Hodge structures that are of split Weil type (relative to
$E$, and with polarization given by $\psi$). We shall show that $\Dsp$ is a certain
hermitian symmetric domain.

We begin by observing that there are essentially $2^{\Qdeg{E}}/2$ many different
choices for the totally imaginary element $\zeta$, up to multiplication by totally
positive elements in $\Fun$. Indeed, if we fix a choice of $i = \sqrt{-1}$, and
define $\phizeta \colon S \to \{0,1\}$ by the rule
\begin{equation} \label{eq:phizeta}
	\phizeta(s) = \begin{cases}
		1 &\text{if $s(\zeta) i > 0$,} \\
		0 &\text{if $s(\zeta) i < 0$,}
	\end{cases}
\end{equation}
then $\phizeta(s) + \phizeta(\sb) = 1$ because $\sb(\zeta) = - s(\zeta)$, and so
$\phizeta$ is a CM-type for $E$. Conversely, one can show that any CM-type is
obtained in this manner.

\begin{lem}
The subset $\Dsp$ of the period domain $D$ is a hermitian symmetric domain; in fact,
it is isomorphic to the product of $\abs{S} = \Qdeg{E}$ many copies of Siegel upper halfspace.
\end{lem}

\begin{proof}
Recall that $V$ is an $E$-vector space of even dimension $d$, and that the Riemann
form $\psi = \TrEQ(\zeta \phi)$ for a split $E$-hermitian form $\phi \colon V \times
V \to E$ and a totally imaginary $\zeta \in \Eun$. The Rosati involution
corresponding to $\psi$ induces complex conjugation on $E$; this means
that $\psi(ev, w) = \psi(v, \bar{e} w)$ for every $e \in E$.

By definition, $\Dsp$ parametrizes all Hodge structures of type $\{(1,0), (0,1)\}$ on $V$
that admit $\psi$ as a Riemann form and are of split Weil type (relative to the CM-field $E$).
Such a Hodge structure amounts to a decomposition
\[
	V \tensor_{\QQ} \CC = V^{1,0} \oplus V^{0,1}
\]
with $V^{0,1} = \overline{V^{1,0}}$, with the following two properties:
\begin{enumerate}[label=(\alph{*}),ref=(\alph{*})]
\item The action by $E$ preserves $V^{1,0}$ and $V^{0,1}$.
	\label{en:prop-a}
\item The form $-i \psi(x, \bar{y}) = \psi \bigl( x, h(i) \bar{y} \bigr)$ is positive
definite on $V^{1,0}$.
	\label{en:prop-b}
\end{enumerate}
Let $S = \Hom(E, \CC)$, and consider the isomorphism
\[
	V \tensor_{\QQ} \CC \xrightarrow{\sim} \bigoplus_{s \in S} V_s, \quad
		v \tensor z \mapsto \sum_{s \in S} v \tensor z,
\]
where $V_s = V \tensor_{E,s} \CC$. Since $V_s$ is exactly the subspace on which $e
\in E$ acts as multiplication by $s(e)$, the condition in \ref{en:prop-a} is
equivalent to demanding that each complex vector space $V_s$ decomposes as $V_s =
V_s^{1,0} \oplus V_s^{0,1}$.

On the other hand, $\phi$ induces a hermitian form $\phi_s$ on each $V_s$, and we
have
\[
	\psi(v,w) = \TrEQ \bigl( \zeta \phi(v,w) \bigr) =
		\sum_{s \in S} s(\zeta) \phi_s(v \tensor 1, w \tensor 1).
\]
Therefore $\psi$ polarizes the Hodge structure $V^{1,0} \oplus V^{0,1}$ if and only if
the form $x \mapsto -s(\zeta) i \cdot \phi_s(x, \bar{x})$ is positive definite on the subspace $V_s^{1,0}$.
Referring to the definition of $\phizeta$ in \eqref{eq:phizeta}, this is
equivalent to demanding that $x \mapsto (-1)^{\phizeta(s)} \phi_s(x, \bar{x})$ be positive definite on
$V_s^{1,0}$.

In summary, Hodge structures of split Weil type on $V$ for which $\psi$ is a Riemann
form are parametrized by a choice of $d/2$-dimensional complex subspaces $V_s^{1,0}
\subseteq V_s$, one for each $s \in S$, with the property that
\[
	V_s^{1,0} \cap \overline{V_s^{1,0}} = \{0\},
\]
and such that $x \mapsto (-1)^{\phizeta(s)} \phi_s(x, \bar{x})$ is positive definite on $V_s^{1,0}$.
Since for each $s \in S$, we have $a_s = b_s = d/2$, the hermitian form $\phi_s$ has
signature $(d/2, d/2)$; this implies that the space
\[
	D_s = \menge{W \in \Grass_{d/2}(V_s)}%
		{\text{$W \cap \overline{W} = \{0\}$
			and $(-1)^{\phizeta(s)} \phi_s(x, \bar{x}) > 0$ for $0 \neq x \in W$}}
\]
is isomorphic to Siegel upper halfspace. The parameter space $\Dsp$ for our Hodge
structures is therefore the hermitian symmetric domain
\[
	\Dsp \simeq \prod_{s \in S} D_s.
\]
In particular, it is a connected complex manifold.
\end{proof}

To be able to satisfy the final condition in Theorem~\ref{thm:Step3}, we need to know
that $\Dsp$ contains Hodge structures of the form $V_0 \tensor_{\QQ} E$. This is the
content of the following lemma.

\begin{lem} \label{lem:product}
With notation as above, there is a rational Hodge structure $V_0$ of weight one, such
that $V_0 \tensor_{\QQ} E$ belongs to $\Dsp$.
\end{lem}

\begin{proof}
Since the pair $(V, \phi)$ is split, there is a totally isotropic subspace $W
\subseteq V$ of dimension $\dim_E W = d/2$. Arguing as in the proof of
Corollary~\ref{cor:split}, we can therefore find a basis $v_1, \dotsc, v_d$ for the
$E$-vector space $V$, with the property that
\begin{align*}
	\phi(v_i, v_{i+d/2}) = \zeta^{-1}  \quad  &\text{for $1 \leq i \leq d/2$,}	\\
	\phi(v_i, v_j) = 0  \quad        &\text{for $\abs{i-j} \neq d/2$.}
\end{align*}
Let $V_0$ be the $\QQ$-linear span of $v_1, \dotsc, v_d$; then we have $V = V_0
\tensor_{\QQ} E$. Now define $V_0^{1,0} \subseteq V_0 \tensor_{\QQ} \CC$ as the
$\CC$-linear span of the vectors $h_k = v_k + i v_{k+d/2}$ for $k=1, \dotsc, d/2$.
Evidently, this gives Hodge structure of weight one on $V_0$, with hence a Hodge
structure on $V = V_0 \tensor_{\QQ} E$. It remains to show that $\psi$ polarizes this
Hodge structure. But we compute that
\begin{align*}
	\psi \left( \sum_{j=1}^{d/2} a_j h_j, i \sum_{k=1}^{d/2} \overline{a_k h_k} \right)
		&= \sum_{k=1}^{d/2} \abs{a_k}^2 \psi \bigl( v_k + i v_{k+d/2}, i (v_k - i v_{k+d/2}) \bigr)
		= 2 \sum_{k=1}^{d/2} \abs{a_k}^2 \psi(v_k, v_{k+d/2}) \\
		&= 2 \sum_{k=1}^{d/2} \abs{a_k}^2 \TrEQ \bigl( \zeta \phi(v_k, v_{k+d/2}) \bigr)
		= 2 \Qdeg{E} \sum_{k=1}^{d/2} \abs{a_k}^2,
\end{align*}
which proves that $x \mapsto \psi(x, i \bar{x})$ is positive definite on the subspace
$V_0^{1,0}$. The Hodge structure $V_0 \tensor_{\QQ} E$ therefore belongs to $\Dsp$ as
desired.
\end{proof}

\begin{proof}[Proof of Theorem~\ref{thm:Step3}]
Let $\theta \colon A \to \hat{A}$ be the polarization on $A$. As before, let
$\mathcal{M}$ be the moduli space of abelian varieties of dimension $d/2$, with
polarization of the same type as $\theta$, and level $3$-structure. Then
$\mathcal{M}$ is a quasi-projective complex manifold, and the period domain
$D$ is its universal covering space (with the Hodge structure $H^1(A)$ mapping to the
point $A$). Let $B \subseteq \mathcal{M}$ be the locus of those abelian varieties
whose endomorphism algebra contains $E$. Note that the original abelian variety $A$
is contained in $B$.  Since every element $e \in E$ is a Hodge class in $\End(A)
\tensor \QQ$, it is clear that $B$ is a Hodge locus; in particular, $B$ is a
quasi-projective variety by the theorem of Cattani-Deligne-Kaplan. As before, we let
$\pi \colon \famA \to B$ be the restriction of the universal family of abelian varieties to $B$.

Now we claim that the preimage of $B$ in $D$ is precisely the set $\Dsp$ of Hodge
structures of split Weil type. Indeed, the endomorphism ring of any Hodge structure
in the preimage of $B$ contains $E$ by construction; since it is also polarized by
the form $\psi$, all the conditions in Definition~\ref{def:split-Weil-type} are
satisfied, and so the Hodge structure in question belongs to $\Dsp$. Because $D$ is the
universal covering space of $\Mmod$, this implies in particular that $B$ is connected
and smooth, hence a quasi-projective complex manifold.

The first two assertions are obvious from the construction, whereas the third follows
from Lemma~\ref{lem:product}. This concludes the proof.
\end{proof}

To complete the proof of Deligne's theorem, we have to show that every split Weil
class is an absolute Hodge class. For this, we argue as follows. Consider the family
of abelian varieties $\pi \colon \famA \to B$ from Theorem~\ref{thm:Step3}. By
Proposition~\ref{prop:split-Hodge}, the space of split Weil classes $\wed{E}{d}
H^1(\famA_t, \QQ)$ consists of Hodge classes for every $t \in B$. The family also
contains an abelian variety of the form $A_0 \tensor_{\QQ} E$, and according to 
Proposition~\ref{prop:point}, all split Weil classes on this particular abelian
variety are absolute. But now $B$ is irreducible, and so Principle~B applies and shows
that for every $t \in B$, all split Weil classes on $\famA_t$ are absolute. This
finishes the third step of the proof, and completely establishes Deligne's theorem.


\end{document}